\documentclass[12pt,a4paper]{article}

\usepackage{amsfonts,amsmath,mathtools,amssymb,mathrsfs}

\usepackage{hyperref}

\usepackage{latexsym}
\usepackage{times}
\usepackage{xcolor}
\newcommand{\dint}{\mathrm{d}}

\author{K\'aroly J. B\"or\"oczky, Ferenc Fodor\thanks{The author was supported by the Hungarian National Research, Development and Innovation Office – NKFIH grant 134814. This research was supported by project TKP2021-NVA-09. Project no. TKP2021-NVA-09 has been implemented with the support provided by the Ministry of Innovation and Technology of Hungary from the National Research, Development and Innovation Fund, financed under the TKP2021-NVA funding scheme.}, 
Daniel Hug\thanks{The author was supported by DFG research grant HU 1874/5-1 (SPP 2265).}}

\title{Strengthened inequalities for the mean width and the $\ell$-norm of origin symmetric convex bodies
\footnote{{\em AMS 2020 subject classification.}
Primary 52A40; Secondary 52A38, 52B12, 26D15.
\newline
{\em Key words and phrases.} Mean width, $\ell$-norm, crosspolytope, cube, extremal problem,  John ellipsoid, L\"owner ellipsoid,  Brascamp--Lieb inequality, mass transportation, stability result, isotropic measure.}}

\newcommand{\proof}{\noindent{\it Proof. }}
\newcommand{\proofbox}{\mbox{ $\Box$}}
\newcommand{\R}{\mathbb{R}}

\DeclareMathOperator{\Id}{I}
\DeclareMathOperator{\conv}{conv}

\newtheorem{lemma}{Lemma}[section]
\newtheorem{theo}[lemma]{Theorem}

\newtheorem{coro}[lemma]{Corollary}

\newtheorem{prop}[lemma]{Proposition}

\allowdisplaybreaks
\date{}

\begin{document}

\maketitle

\begin{abstract}
Barthe, Schechtman and Schmuckenschl\"ager proved that the cube maximizes the mean width of symmetric convex bodies whose John ellipsoid (maximal volume ellipsoid contained in the body)  is the Euclidean unit ball,  and the regular crosspolytope minimizes the mean width of symmetric convex bodies whose L\"owner ellipsoid is the Euclidean unit ball.
Here we prove close to be optimal stronger stability versions of these results, together with their counterparts about the $\ell$-norm based on Gaussian integrals. We also consider related stability results for the mean width and the $\ell$-norm of the convex hull of the support of even  isotropic measures on the unit sphere.
\end{abstract}

\section{Introduction}

Geometric inequalities and extremal problems constitute a central topic in geometry and geometric analysis. Perhaps the best known example is the isoperimetric inequality which states that Euclidean balls have smallest surface area among  sets of given (finite) volume in Euclidean space $\R^n$. If we restrict to compact convex sets with nonemtpy interior (convex bodies), then Euclidean balls are the only minimizers. Another important example is the Urysohn inequality according to which Euclidean balls minimize the mean width of convex bodies of given volume.  While for these two classical examples the extremizers (Euclidean balls) are rotationally symmetric, for some other extremal problems simplices, cubes and crosspolytopes naturally arise as extremizers.

In the following, we mainly focus on geometric inequalities and extremal problems for origin symmetric ($o$-symmetric) convex bodies (see \cite{BFH21} for the non-symmetric setting) in Euclidean space $\R^n$, $n\ge 2$, with scalar product $\langle\cdot\,,\cdot\rangle$ and norm $\|\cdot\|$.   
We write $B^n_p$ to denote the unit ball of the $l_p$ norm in $\R^n$ for $p\in[1,\infty]$, in particular $B^n_1$ is a regular crosspolytope inscribed into $B^n_2$, and $(B^n_1)^\circ=B^n_\infty$ is a cube circumscribed around $B^n_2$,  where $K^\circ=\{x\in\R^n:\langle x,y\rangle\leq 1\;\forall\, y\in K\}$  
is the polar dual of an origin symmetric convex body $K\subset\R^n$.

An intriguing geometric extremal problem for which cubes and crosspolytopes are extremizers has been discovered and studied much more recently. Let $K$ denote some origin symmetric convex body in $\R^n$. It is well known that there exists a unique ellipsoid of maximal volume contained in $K$ (the John ellipsoid of $K$), and a unique ellipsoid of minimal volume containing $K$ (the L\"owner ellipsoid of $K$). It has been proved by Ball \cite{Bal91a} that cubes maximize the volume of $K$
given the volume of the John ellipsoid of $K$, which can be paraphrased by saying that cubes determine the extremal ``inner" volume ratio for origin symmetric convex bodies. For the dual problem, Barthe \cite{Bar98} showed that
crosspolytopes minimize  the volume of $K$
given the volume of the L\"owner ellipsoid of $K$, hence crosspolytopes determine the extremal ``outer" volume ratio (see also Lutwak, Yang, Zhang \cite{Lutwak0,Lutwak1}). In all these cases, the understanding of the equality cases essentially relies on Barthe's fundamental work  \cite{Bar98} on the Brascamp--Lieb inequality and its reverse inequality.

Since the mean width plays a key role in the present work, we provide an explicit definition. Let  $\kappa_n$ denote the volume of the unit ball $B^n_2$ and $S^{n-1}$ its boundary (the Euclidean unit sphere). For a convex body $K\subset \R^n$, the support function $h_K:\R^n\to\R$ of $K$ is defined by
$h_K(x)=\max_{y\in K}\langle x,y\rangle$ for $x\in\R^n$. Then the mean width of $K$ is given by
$$
W(K)=\frac1{n\kappa_n}\int_{S^{n-1}}(h_K(u)+h_K(-u))\,\dint u,
$$
where the integration is with respect to the $(n-1)$-dimensional spherical Lebesgue measure. As a functional on convex bodies, the mean width is isometry invariant, continuous, additive (a valuation) and (positively) homogeneous of degree one. Moreover the mean width is characterized by these properties.

In addition to the (translation invariant) mean width we consider the $\ell$-norm of origin symmetric convex bodies.
To define the latter,  for an $o$-symmetric convex body $K\subset\R^n$, we set
$$
\|x\|_K=\min\{t\geq 0:\,x\in tK\}, \qquad  x\in\R^n.
$$
Let $\gamma_n$ denote the standard Gaussian measure in $\R^n$, which has the density function
$x\mapsto \frac1{\sqrt{2\pi}^{n}}\,e^{-{\|x\|^2}/{2}}$, $x\in\R^n$,  with respect to Lebesgue measure. Then the $\ell$-norm of $K$ is defined by
$$
\ell(K)=\int_{\R^n}\|x\|_K\,\gamma_n(\dint x)=\mathbb{E} \| X\|_K,
$$
where $X$ is a Gaussian random vector with distribution $\gamma_n$. Using polar coordinates, the relation  $h_{K^\circ}=\|\cdot\|_K$ and denoting by $\Gamma(\cdot)$ the Gamma function, we get
\begin{equation}
\label{mean-ell}
\ell(K)=\mbox{$\frac{\ell(B^n_2)}2$}  W(K^\circ)
\end{equation}
with $\ell(B^n_2)=\sqrt{2}\Gamma(\frac{n}{2})^{-1}\Gamma(\frac{n+1}{2})$, hence
$$
\lim_{n\to\infty}\frac{\ell(B^n_2)}{\sqrt{n}}=1.
$$
Clearly, the $\ell$-norm of $K$ can be rewritten in the form  
\begin{equation}
\label{ellGaussianint}
\ell(K)=\int_{0}^\infty\mathbb{P}(\|X\|_K>t)\, \dint t
=\int_0^\infty(1-\gamma_n(tK))\,\dint t.
\end{equation}

 The following Theorem~\ref{meanw-sym}   for the $\ell$-norm of symmetric convex bodies was first stated by Schechtman, Schmuckenschl\"ager \cite{ScS95} (without statements concerning the equality cases). Part (i) is only briefly mentioned \cite[p.~270]{ScS95}), an explicit proof including a characterization of the equality case was provided by Barthe \cite[Theorem 2]{Bar98b}.  Part (ii) was  proved by Schechtman, Schmuckenschl\"ager   \cite[Proposition 4.11]{ScS95} as an application of the Brascamp--Lieb inequality. The equality case follows from Barthe's general analysis of the equality conditions in the Brascamp--Lieb inequality \cite{Bar98}. The non-symmetric cases of these statements are treated by Barthe \cite{Bar98b} and Schmuckenschl\"ager \cite{Sch99} (see \cite{BFH21} for strengthened inequalities in the non-symmetric case).

\begin{theo}[Barthe '98, Schechtman, Schmuckenschl\"ager '95]
\label{meanw-sym}
Let $K$ be an origin symmetric convex body in $\R^n$.
\begin{enumerate}
\item[{\rm (i)}] If $B^n_2\supset K$ is the L\"owner ellipsoid of $K$, then $\ell(K)\leq \ell(B^n_1)$
and $W(K)\geq W(B^n_1)$. Equality holds in either case if and only if $K$ is a regular crosspolytope.

\item[{\rm (ii)}] If $B^n_2\subset K$ is the John ellipsoid of $K$, then $W(K)\leq W(B^n_\infty)$ and $\ell(K)\geq \ell(B^n_\infty)$.
  Equality holds in either case if and only if $K$ is a cube.
\end{enumerate}
\end{theo}

It follows from (\ref{mean-ell}) and the duality of L\"owner and John ellipsoids that the first statement of (i) is equivalent to
the first statement of (ii), and the second statement of (i) is equivalent to
the second statement of (ii).

Let us briefly discuss the range of $W(K)$ (and thus that of $\ell(K)$) in Theorem~\ref{meanw-sym}. Note that
$W(K)=\frac{2\kappa_{n-1}}{n\kappa_n}\cdot V_1(K)$ ({\it cf.} Schneider \cite[(5.31) and (5.57)]{Sch14} or \cite[(3.18)]{HW20}), where $V_1(K)$ is the first intrinsic volume of $K$, 
and $\frac{\kappa_{n-1}}{\kappa_n}\sim \sqrt{\frac{n}{2\pi}}$, as $n$ tends to infinity.
If $K\subset\R^n$ is an $o$-symmetric convex body  whose L\"owner ellipsoid is $B_2^n$,
then the monotonicity of the mean width and Theorem~\ref{meanw-sym} (i) yield
$$
W(B^n_1)\leq W(K)\leq W(B^n_2)=2,
$$
where Lemma~3.1 in B\"or\"oczky, Henk \cite{BoH99} yields that
$$
W(B^n_1)\sim \sqrt{\frac{2\ln n}n}\mbox{ \ as $n\to\infty$}.
$$
In addition, $V_1(B^n_\infty)=2n$ as $V_1$ is additive. If $K$ is a convex body in $\R^n$ whose John ellipsoid is $B^n_2$,
then
$$
2=W(B^n_2)\leq W(K)\leq W(B^n_\infty),
$$
where $W(B^n_\infty)\sim \sqrt{8n/\pi}$. As a consequence, we also have $\ell(B^n_\infty)\sim \sqrt{\frac{\ln n}{2}}$  as $n\to\infty$ and $\ell(B^n_1)= \sqrt{\frac{2}{\pi}}\cdot n$ for $n\ge 2$.

The main goal of this paper is to prove a close to optimal stability version of Theorem~\ref{meanw-sym}. In the stronger version of Theorem~\ref{meanw-sym}, closeness of two compact sets $X,Y\subset\R^n$ is measured in terms of their Hausdorff distance  
$$
\delta_{\rm H}(X,Y)=\max\{\max_{y\in Y}d(y,X),\max_{x\in X}d(x,Y)\},
$$
where $d(x,Y)=\min\{\|x-y\|:y\in Y\}$ is the distance of $x$ from $Y$. The Hausdorff distance defines a metric on the set of non-empty compact subsets of $\R^n$. Let $O(n)$ denote the orthogonal group of $\R^n$.
Since we are interested in the Hausdorff distance up to orthogonal transformations, we also consider the metric
$$
\delta_{\rm HO}(X,Y)=\min\{\delta_{\rm H}(X,\Phi Y):\,\Phi\in O(n)\}.
$$

We start with the case when $B^n_2$ is the L\"owner ellipsoid of a convex body $K\subset B^n_2$, and then state
the stability result when $B^n_2$ is the John ellipsoid of a convex body $K\subset \R^n$.

\begin{theo}
\label{Lowner-stab}
If $B^n_2$ is the L\"owner ellipsoid of an origin symmetric convex body $K\subset B^n_2$, then
\begin{align}
\label{Lowner-stab-ell}
\ell(K)&\leq \left(1-\gamma\cdot\delta_{\rm HO}(K,B^n_1)\right)\ell(B^n_1),\\
\label{Lowner-stab-W}
W(K)&\geq \left(1+\gamma\cdot\delta_{\rm HO}(K,B^n_1)^{4n}\right)W(B^n_1),
\end{align}
  where $\gamma=n^{-cn^2}$ with an absolute constant $c>0$.
\end{theo}

\begin{theo}
\label{John-stab}
If $B^n_2$ is the John ellipsoid of an origin symmetric convex body $K\subset \R^n$, then
\begin{align}
\label{John-stab-ell}
\ell(K)&\geq \left(1+\gamma\cdot\delta_{\rm HO}(K,B^n_\infty)^{4n}\right)\ell(B^n_\infty),\\
\label{John-stab-W}
W(K)&\leq \left(1-\gamma\cdot\delta_{\rm HO}(K,B^n_\infty)\right)W(B^n_\infty),
\end{align}
 where $\gamma=n^{-cn^2}$ with an absolute constant $c>0$.
\end{theo}

Concerning the optimality of Theorem~\ref{John-stab},  let $K_\varepsilon$, for sufficiently small $\varepsilon>0$, be obtained from the cube $B^n_\infty$ by cutting off each vertex $v$ of $B^n_\infty$ by the half space
$\{x\in\R^n:\langle x,v\rangle\leq (1-\varepsilon)\langle v,v\rangle\}$. The resulting body satisfies that
$\delta_{\rm HO}(K_\varepsilon,B^n_\infty)\geq \gamma_1\varepsilon$, 
$\ell(K_\varepsilon)\leq \left(1+\gamma_2\varepsilon^{n}\right)\ell(B^n_\infty)$ and
$W(K_\varepsilon)\geq \left(1-\gamma_3\varepsilon\right)W(B^n_\infty)$ for some $\gamma_1,\gamma_2,\gamma_3>0$ depending on $n$; therefore, 
\begin{itemize}
\item the estimate \eqref{John-stab-W} is optimal up to the factor depending on $n$, and 
\item the exponent $4n$ of $\delta_{\rm HO}(K,B^n_\infty)$ in \eqref{John-stab-ell} cannot be replaced by anything smaller than $n$.
\end{itemize}

 For the optimality of Theorem~\ref{Lowner-stab}, we consider the polar
$K_\varepsilon^\circ$ for sufficiently small $\varepsilon>0$, and hence
the estimate \eqref{Lowner-stab-ell} is optimal up to the factor depending on $n$, and the exponent $4n$ of $\delta_{\rm HO}(K,B^n_1)$ in \eqref{Lowner-stab-W} cannot be replaced by anything smaller than $n$.

Finally, in the statement of relation \eqref{Lowner-stab-W} the factor $\gamma$ can be chosen of the order $n^{-cn}$ with an absolute constant $c>0$ and $\delta_{\rm HO}(K,B^n_1)^{4n}$ can be replaced by $\delta_{\rm volO}(K,B^n_1)^{4}$, where for origin symmetric convex bodies $X,Y\subset\R^n$ we define
$$
\delta_{\rm volO}(X,Y)=\min\{\delta_{\rm vol}(X,\Phi Y):\,\Phi\in O(n)\}
$$
and $\delta_{\rm vol}$ is the symmetric (volume) difference metric. Similar comments apply to relation \eqref{John-stab-ell}.

An important concept underlying the proof of Theorem~\ref{meanw-sym} is the notion of an isotropic measure on the unit sphere. 
A Borel measure $\mu$ on the unit sphere $S^{n-1}$ is called isotropic ({\it cf.}~\cite{GianPapa1999,Lutwak1}) if
\begin{equation}
\label{isotropic-def}
\int_{S^{n-1}}u\otimes u\, \mu(\dint u)=\Id_n,
\end{equation}
where $\Id_n$ is the {identity map (or the identity matrix)}. 
Equating traces of the two sides of (\ref{isotropic-def}), we get 
$\mu(S^{n-1})=n$.
If $\mu$ is an even isotropic measure on $S^{n-1}$, then the cardinality of its support satisfies $|{\rm supp}\,\mu|\geq 2n$, with equality if and only if $\mu$ is concentrated on the vertices of some regular crosspolytope and each vertex has  measure $\frac12$. 
If $u_1,\ldots,u_k\in S^{n-1}$ and $c_1,\ldots,c_k>0$ satisfy  $u_i\neq \pm u_j$ for $i\neq j$ and  
$$
\sum_{i=1}^kc_iu_i\otimes u_i=\Id_n,
$$
then the discrete  measure $\mu$ on $S^{n-1}$ with support $\{\pm u_1,\ldots,\pm u_k\}$ and $\mu(\{u_i\})=\mu(\{-u_i\})=c_i/2$ for
$i=1,\ldots,k$ is isotropic and even. If $k=n$, $u_1,\ldots,u_n$ form an orthonormal basis of $\R^n$ and each $c_i=1$, then such a measure is called a {\it cross measure}, and is characterized by the properties that it is an even isotropic measure whose support consists of the vertices of a regular crosspolytope.

We recall that isotropic measures on $\R^n$ play a central role in the
KLS conjecture \cite{KLM95} and in the analysis of Bourgain's hyperplane conjecture (slicing problem) (see, for instance, \cite{BCE13,GuM11,Kla09,AGM15,AGB15}). 
The relevance of isotropic measures on $S^{n-1}$ in the present context is due to Ball's crucial insight that John's characteristic condition \cite{Joh37,Joh48} (see Section \ref{secLoewnerBall}) for a convex body to have the unit ball as its John or L\"owner ellipsoid
 (see \cite{Bal89,Bal91a}) can be used to give the Brascamp--Lieb inequality a convenient form which is perfectly  suited  for many
 geometric applications 
(see Section~\ref{secBrascamp-Lieb}).

We write ${\rm conv}\,X$ to denote the convex hull of a set $X\subset \R^n$, and for an even isotropic measure $\mu$ on $S^{n-1}$, we consider the $o$-symmetric convex bodies
$$
Z_\infty(\mu)={\rm conv}\,{\rm supp}\,\mu\subset B^n_2\mbox{ \ and \ }Z^*_\infty(\mu)=Z_\infty(\mu)^\circ\supset B^n_2.
$$
In particular, if $\nu$ is a cross measure on $S^{n-1}$, then
$$
Z_\infty(\nu)=B^n_1\mbox{ \ and \ }Z^*_\infty(\nu)=B^n_\infty.
$$
Li, Leng \cite{LiL12} proved the following version of Theorem~\ref{meanw-sym}.  The result 
can be obtained as an immediate consequence of Theorem \ref{meanw-sym} and Lemma \ref{JohnLimited}.

 \begin{theo}[Li, Leng \cite{LiL12}]
\label{meanw-iso-measure}
Let $\nu$ be a cross measure on $S^{n-1}$. If  $\mu$ is an even isotropic measure, then 
\begin{enumerate}
\item[{\rm (i)}]  $\ell(Z_\infty(\mu))\leq \ell(Z_\infty(\nu))$
and $W(Z_\infty(\mu))\geq W(Z_\infty(\nu))$, and  equality holds in either case if and only if $\mu$ is a cross measure.

\item[{\rm (ii)}] $\ell(Z^*_\infty(\mu))\geq \ell(Z^*_\infty(\nu))$
and $W(Z^*_\infty(\mu))\leq W(Z^*_\infty(\nu))$, and  equality holds in either case if and only if $\mu$ is a cross measure.
\end{enumerate}
\end{theo}

In order to state a stability version of Theorem ~\ref{meanw-iso-measure}, we need the notion of Wasserstein distance $\delta_{\rm W}(\mu,\nu)$ of two isotropic measures $\mu$ and $\nu$  on $S^{n-1}$ (also called the Kantorovich--Monge--Rubinstein distance). To define it, we write
$\angle(v,w)$ to denote the angle of $v, w\in S^{n-1}$; that is, the geodesic distance of $v$ and $w$ on $S^{n-1}$. Let
 ${\rm Lip}_1(S^{n-1})$ denote the family of Lipschitz functions with Lipschitz constant $1$; namely, $f:\,S^{n-1}\to\R$ is in ${\rm Lip}_1(S^{n-1})$ if 
$\|f(x)-f(y)\|\leq \angle(x,y)$ for
$x,y\in S^{n-1}$. Then
$$
\delta_{\rm W}(\mu,\nu)=\max\left\{\int_{S^{n-1}}f\,d\mu-\int_{S^{n-1}}f\,d\nu:\,f\in {\rm Lip}_1(S^{n-1})\right\}.
$$
Since we allow rotation of one of the measures, the actual notion of distance that we use is
$$
\delta_{\rm WO}(\mu,\nu)=\min\left\{\delta_{\rm W}(\mu,\Phi_*\nu):\,\Phi\in O(n)\right\}
$$
where $\Phi_*\nu$ denotes the pushforward of $\nu$ by $\Phi: S^{n-1}\to S^{n-1}$. 
In turn, we have the following stronger form of Theorem ~\ref{meanw-iso-measure}.

\begin{theo}
\label{meanw-iso-measure-stab}
Let $\mu$ be an even isotropic measure  on $S^{n-1}$.
There exist an absolute constant $c>0$ and a cross measure $\nu$ on $S^{n-1}$ such that for $\gamma=n^{-cn^2}$,  for (a) and (d), and $\gamma=n^{-cn}$,  for (b) and (c), and 
$\tilde{\gamma}=\frac{\gamma}{30n^3}$, we have
\begin{enumerate}
\item[{\rm (a)}] $\begin{array}[t]{rcl}
\ell(Z_\infty(\mu))&\leq& \left(1-\gamma\cdot\delta_{\rm HO}({\rm supp}\,\mu,{\rm supp}\,\nu)\right)\ell(Z_\infty(\nu))\\[1ex]
&\leq &\left(1-\tilde{\gamma}\cdot\delta_{\rm WO}(\mu,\nu)\right)\ell(Z_\infty(\nu)).
\end{array}
$ 
\item[{\rm (b)}] $\begin{array}[t]{rcl}
W(Z_\infty(\mu))&\geq& \left(1+\gamma\cdot\delta_{\rm HO}({\rm supp}\,\mu,{\rm supp}\,\nu)^{4}\right)W(Z_\infty(\nu))\\[1ex]
&\geq &\left(1+\tilde{\gamma}\cdot\delta_{\rm WO}(\mu,\nu)^{4}\right)W(Z_\infty(\nu)).
\end{array}
$ 
\item[{\rm (c)}] $\begin{array}[t]{rcl}
\ell(Z^*_\infty(\mu))&\geq& \left(1+\gamma\cdot\delta_{\rm HO}({\rm supp}\,\mu,{\rm supp}\,\nu)^{4}\right)\ell(Z^*_\infty(\nu))\\[1ex]
&\geq &\left(1+\tilde{\gamma}\cdot\delta_{\rm WO}(\mu,\nu)^{4}\right)\ell(Z_\infty(\nu)).
\end{array}
$ 
\item[{\rm (d)}] $\begin{array}[t]{rcl}
W(Z^*_\infty(\mu))&\leq& \left(1-\gamma\cdot\delta_{\rm HO}({\rm supp}\,\mu,{\rm supp}\,\nu)\right) W(Z^*_\infty(\nu))\\[1ex]
&\leq &\left(1-\tilde{\gamma}\cdot\delta_{\rm WO}(\mu,\nu)\right)W(Z_\infty(\nu)).
\end{array}
$ 
\end{enumerate}
\end{theo}

The main task of the present work is to estimate the $\ell$-norm. 
In Section~\ref{secLoewnerBall} we deal with the case when the L\"owner ellipsoid is a ball. This part of the argument avoids the use of Barthe's reverse form of the Brascamp--Lieb inequality and proceeds by a more direct reasoning. To estimate the $\ell$-norm when the John ellipsoid is a ball, first we review the use of optimal transport and the Brascamp--Lieb inequality in Section~\ref{secBrascamp-Lieb}, and secondly we provide some crucial estimates for the transport functions arising in this context in Section~\ref{secAuxiliaryJohn}.
Stability of the $\ell$-norm around the cube when the John ellipsoid is a ball is verified in Section~\ref{secJohnBall} (see Theorem \ref{ell-sym-John-stab}), and we prove Theorem~\ref{Lowner-stab}
and Theorem~\ref{John-stab} in 
 Section~\ref{secTheorem1314}.
Finally, Theorem~\ref{meanw-iso-measure-stab} is established in 
Section~\ref{secTheorem15} via Theorem \ref{theo:theorem 7.4} and Theorem \ref{theo:theorem 7.7}.

Our arguments are partly based on the rank one {geometric} Brascamp--Lieb  inequality 
 and its stability version in a special case (see Section~\ref{secJohnBall}). No quantitative stability version of the Brascamp--Lieb inequality (or of the reverse Brascamp--Lieb inequality) is  known in general
(see Bennett, Bez, Flock, Lee \cite{BBFL18} for a certain weak stability version for higher ranks).
 On the other hand, in the case
of the Borell--Brascamp--Lieb inequaliy (see Borell \cite{Bor75},  Brascamp, Lieb \cite{BrL76} and Balogh, {Krist\'aly} \cite{BaK18}),
 stability versions were obtained by Ghilli, Salani \cite{GhS17} and Rossi, Salani \cite{RoS171}.

\section{The $\ell$-norm when the L\"owner ellipsoid is a ball}
\label{secLoewnerBall}

Theorem~\ref{ell-sym-Loewner} below is  part (i) of Theorem~\ref{meanw-sym}. We start by recalling the argument of Barthe (see \cite[Theorem 2]{Bar98b}) to prove Theorem~\ref{ell-sym-Loewner}. 

In order to efficiently use the hypothesis that the unit ball is the John (or L\"owner) ellipsoid of a symmetric convex body $K$, 
John's \cite{Joh37,Joh48} characteristic condition is employed. It states  that $B^n_2$ is the John ellipsoid (or L\"owner ellipsoid) of an $o$-symmetric convex body $K$ if and only if $B^n_2\subset K$ (or  $K\subset B^n_2$), and there exist distinct unit vectors $u_1,\ldots,u_k\in \partial K\cap S^{n-1}$ and $c_1,\ldots,c_k>0$ such that $u_i\neq \pm u_j$ for $i\neq j$ and
\begin{equation}
\label{John-iso}
\sum_{i=1}^kc_iu_i\otimes u_i=\Id_n,
\end{equation}
with the proof of the equivalence completed by Ball \cite{Bal92} (see also  \cite{GrS05}).

It follows that $B^n_2$ is the John ellipsoid of an $o$-symmetric convex body $K\subset\R^n$ if and only if $B^n_2$ is the  L\"owner ellipsoid of $K^\circ$.

Furthermore, there exists some $m$ element subset $I\subset\{1,\ldots,k\}$ with
$m\leq\frac{n(n+1)}2$ and some $\tilde{c}_i>0$ for $i\in I$ such that
$\sum_{i\in I}\tilde{c}_iu_i\otimes u_i=\Id_n$.

In the following we fix an orthonormal basis $e_1,\ldots,e_n$ of $\R^n$. We also note that
if $u_1,\ldots,u_k\in S^{n-1}$ and $c_1,\ldots,c_k>0$
satisfy 
\begin{equation}\label{John-matrix}
\sum_{i=1}^k c_i u_i \otimes u_i=\Id_n,
\end{equation}
then
this data and any $x\in\R^n$ satisfy
\begin{eqnarray}
\label{John-x-1}
x&=&\sum_{i=1}^kc_i\langle x,u_i\rangle\,u_i,\\
\label{John-x2-1}
\|x\|_2^2&=&\sum_{i=1}^kc_i\langle x,u_i\rangle^2,\\
\label{John-sum-ci}
\sum_{i=1}^kc_i&=&n,\\
\label{John-ci-upper}
c_j&\leq & 1\mbox{ \ \ for }j=1,\ldots,k.
\end{eqnarray}
Here \eqref{John-x-1} and \eqref{John-x2-1} directly follow from \eqref{John-matrix},
\eqref{John-sum-ci}
follows from equating traces, and \eqref{John-x2-1} yields \eqref{John-ci-upper} by taking $x=u_j$.

\begin{theo}[Barthe, Schechtman, Schmuckenschl\"ager]
\label{ell-sym-Loewner}
If $K\subset\R^n$ is an origin symmetric convex body such that $B^n_2\supset K$ is the L\"owner ellipsoid of $K$, then $\ell(K)\leq \ell(B^n_1)$. Equality holds if and only if $K$ is a regular crosspolytope inscribed to $B^n_2$. 
\end{theo}
\begin{proof}
According to John's  characteristic condition, there exist distinct unit vectors $u_1,\ldots,u_k\in \partial K\cap S^{n-1}$ and $c_1,\ldots,c_k>0$ such that \eqref{John-matrix} holds. 
Moreover, we can assume that $u_i\neq -u_j$ for $i\neq j$. 
We deduce from $\sum_{i=1}^k c_i=n$ ({\it cf.} \eqref{John-sum-ci})
and 
$$\int_{\R^n}|\langle x,e_1\rangle|\,\gamma_n(\dint x)=
\int_{\R^n}|\langle x,v\rangle|\,\gamma_n(\dint x)$$
for any $v\in S^{n-1}$
that
\begin{align}
\nonumber
\ell(B^n_1)&=\int_{\R^n}\|x\|_{B^n_1}\,\gamma_n(\dint x)=
\int_{\R^n}\sum_{i=1}^n|\langle x,e_i\rangle|\,\gamma_n(\dint x)=n\int_{\R^n}|\langle x,e_1\rangle|\,\gamma_n(\dint x)\\
\label{ellB1uici}
&=\left(\sum_{i=1}^kc_i\right)\int_{\R^n}|\langle x,e_1\rangle|\,\gamma_n(\dint x)=
\int_{\R^n}\sum_{i=1}^kc_i|\langle x,u_i\rangle|\,\gamma_n(\dint x).
\end{align}

The convex hull $M\subset K$ of $\pm u_1,\ldots,\pm u_k$  is an $o$-symmetric polytope, and
\eqref{John-x-1} yields that
\begin{equation}
\label{xnormKM}
\|x\|_K\leq \|x\|_M=\inf\left\{\sum_{i=1}^k|\alpha_i|:\,x=\sum_{i=1}^k\alpha_iu_i\right\}\leq
\sum_{i=1}^kc_i|\langle x,u_i\rangle|.
\end{equation}
It follows from \eqref{ellB1uici} that
\begin{align*}
\ell(K)=&\int_{\R^n}\|x\|_K\,\gamma_n(\dint x)\leq
\int_{\R^n}\|x\|_M\,\gamma_n(\dint x)\\
\nonumber
\leq&
\int_{\R^n}\sum_{i=1}^kc_i|\langle x,u_i\rangle|\,\gamma_n(\dint x)=\ell(B^n_1).
\end{align*}

Next we assume that $\ell(K)=\ell(B^n_1)$. Then equality holds in \eqref{xnormKM}, hence $K=M$, and equality holds in the right inequality of \eqref{xnormKM}. Applying this fact to $x=u_j$ and the estimate $|\langle u_j,u_i\rangle |\le 1$ yield
$$
1\ge \sum_{i=1}^kc_i|\langle u_j,u_i\rangle |\ge  \sum_{i=1}^kc_i \langle u_j,u_i\rangle^2=1.
$$
This shows that $|\langle u_i,u_j\rangle|\in \{0,1\}$ for $i,j\in\{1,\ldots,k\}$. Therefore, we have $k=n$, $u_1,\ldots,u_n$ is an orthonormal basis and $c_1=\cdots=c_n=1$. \hfill\proofbox
\end{proof}

\medskip 

In order to have a stability version of Theorem~\ref{ell-sym-Loewner}, we need some observations on a system of vectors satisfying
\eqref{John-matrix}. We observe that setting $v_i=\sqrt{c_i}\,u_i$ in \eqref{John-matrix}, we have
$\sum_{i=1}^k v_i \otimes v_i=\Id_n$. We deduce the following from the Cauchy--Binet formula.

\begin{lemma}
\label{ciuibig}
If $v_1,\ldots,v_k\in \R^n$
satisfy $\sum_{i=1}^k v_i \otimes v_i={\rm Id}_n$, then
there exist $1\leq i_1<\ldots<i_n\leq k$ such that
$$
\det[v_{i_1},\ldots,v_{i_n}]^2\geq \binom{k}{n}^{-1}.
$$
\end{lemma}

For non-zero vectors $v$ and $w$, we write
$\angle(v,w)$ to denote their angle, that is, the geodesic distance of the unit vectors $\|v\|^{-1}v$ and $\|w\|^{-1}w$ on the unit sphere. The following lemma is a variant of \cite[Lemma 3.2]{BoH17}, which follows by a similar argument as in \cite{BoH17}.

\begin{lemma}
\label{almostort0}
Let  $u_1,\ldots,u_k\in S^{n-1}$ 
and $c_1,\ldots,c_k>0$, $k>n$,
satisfy \eqref{John-matrix}. Let 
$0<\eta\leq 1/(3n)$. If for any $j\in\{n+1,\ldots,k\}$ there exists some 
$i\in\{1,\ldots,n\}$ satisfying
$|\langle u_i,u_j\rangle|\geq \cos\eta$, then
 there exist an orthonormal basis $w_1,\ldots,w_n$ and $\xi_j\in\{-1,1\}$ for $j=1,\ldots,k$ such that
$$
\delta_{\rm H}\left(\{w_1,\ldots,w_n\},\{\xi_1u_1,\ldots,\xi_k u_k\}\right)<4\sqrt{n}\,\eta.
$$
\end{lemma}
\begin{proof} We partition the index set $\{1,\ldots,k\}$ into sets
${\cal V}_1,\ldots,{\cal V}_n$
such that  $i\in{\cal V}_i$ for $i=1,\ldots,n$, and  if $j\ge n+1$ and $j\in{\cal V}_i$ for some $i\in\{1,\ldots,n\}$, then $|\langle u_i,u_j\rangle|\geq \cos\eta$.
For $i=1,\ldots,n$, \eqref{John-x2-1} yields
$$
1=\|u_i\|^2\geq
\sum_{j\in{\cal V}_i} c_j\langle u_i,u_j\rangle^2\geq
\sum_{j\in{\cal V}_i} c_j(\cos\eta)^2,
$$
and hence
\begin{equation}
\label{Vicj}
\sum_{j\in{\cal V}_i} c_j\leq (\cos\eta)^{-2}.
\end{equation}
For $j\in{\cal V}_i$, we may replace $u_j$ with $-u_j$ if necessary to ensure
$\langle u_i,u_j\rangle\geq 0$, and hence
\begin{equation}
\label{uiujanglesmall}
\angle (u_i,u_j)\leq \eta\mbox{ \ \ \ for }j\in{\cal V}_i.
\end{equation}

For $i=1,\ldots,n$, let $\tilde{w}_i\in S^{n-1}$
 be orthogonal to $u_m$, $m\in\{1,\ldots,n\}\setminus\{i\}$,
and satisfy $\langle \tilde{w}_i,u_i\rangle\ge 0$.
In addition, let
 $\alpha_i\in[0,\pi/2]$ be defined by
$$
\cos\alpha_i=\max\{|\langle \tilde{w}_i,u_j\rangle|:\,j\in{\cal V}_i\}.
$$

We aim at bounding $\alpha_i$ from above. For a fixed $i\in\{1,\ldots,n\}$, we observe that
 if $j\in\mathcal{V}_i$, then $\langle \tilde{w}_i,u_j\rangle^2 \le   \cos^2\alpha_i $, and
if  $j\in\mathcal{V}_m$ for some $m\in\{1,\ldots,n\}\setminus\{i\}$, then
$\angle (u_m,u_j)\leq  \eta$ and $\langle \tilde{w}_i,u_m\rangle=0$ imply
$$
\langle u_j, \tilde{w}_i\rangle^2+\cos^2\eta\le \langle u_j,\tilde{w}_i \rangle^2+\langle u_j,u_m\rangle^2\le \|u_j\|^2=1,
$$
and hence
$\langle \tilde{w}_i,u_j\rangle^2 \le  \sin^2 \eta$.
Using these facts and (\ref{Vicj}), we deduce
\begin{align*}
\sum_{j\in{\cal V}_m}c_j \langle \tilde{w}_i,u_j\rangle^2
&\leq\sin^2\eta\sum_{j\in{\cal V}_m} c_j\leq
\frac{\sin^2\eta}{\cos^2\eta},\qquad \mbox{for $m\in\{1,\ldots,n\}\setminus\{i\}$},\\
\sum_{j\in{\cal V}_i} c_j\langle \tilde{w}_i,u_j\rangle^2&\leq
\cos^2\alpha_i\sum_{j\in{\cal V}_i} c_j\leq
\frac{\cos^2\alpha_i}{\cos^2\eta}.
\end{align*}
It follows from \eqref{John-x2-1} that
$$
1=\| \tilde{w}_i\|^2=\sum_{j=1}^kc_j \langle \tilde{w}_i,u_j\rangle^2\leq
 \frac{(n-1)\sin^2\eta}{\cos^2\eta}+\frac{\cos^2\alpha_i}{\cos^2\eta},
$$
and hence
$$
\sin^2\alpha_i=1-\cos^2\alpha_i\leq 1-\cos^2\eta+(n-1)\sin^2\eta=n\sin^2\eta.
$$
This shows that 
\begin{equation}\label{bound1}
\left(\frac{2}{\pi}\right)^2\alpha_i^2\le n\eta^2\le  \frac{1}{9n},
\end{equation}
thus $0\le\alpha_i\le\frac{\pi}{6\sqrt{n}}$. Since $\eta<1/(3n)$, we also get $\alpha_i+\eta< \pi/2$. By the definition of $\alpha_i$, there is some $j\in\mathcal{V}_i$ such that $\cos\alpha_i=\pm\langle \tilde{w}_i,u_j\rangle$. First, suppose that $\cos\alpha_i=- \langle \tilde{w}_i,u_j\rangle$. Then $\angle(\tilde w_i,-u_j)=\alpha_i$, hence \eqref{uiujanglesmall} yields
$$\angle(\tilde w_i,-u_i)\le \angle(\tilde w_i,-u_j)+\angle(-u_j,-u_i)\le \alpha_i+\eta<\frac{\pi}{2},
$$
which is a contradiction to $\langle\tilde w_i,u_i\rangle\ge 0$. This shows that $\cos\alpha_i=  \langle\tilde{w}_i,u_j\rangle$ and therefore $\angle(\tilde w_i,u_j)=\alpha_i$. As before, \eqref{uiujanglesmall} now yields
\begin{equation}
\label{alphai+}
\angle(\tilde{w}_i,u_i)\leq \alpha_i+\eta.
\end{equation}

Therefore,  \eqref{alphai+}, \eqref{bound1} 
and $ \eta<1/(3\sqrt{n})$ imply
\begin{equation}\label{remref3}
\angle(\tilde{w}_i,u_i)\leq \alpha_i+\eta\leq 2\sqrt{n}\,\eta+\eta<3\sqrt{n}\,\eta<1, \qquad i=1,\ldots,n.
\end{equation}
Supposing that $u_i$ is in the linear hull of the vectors $u_m$, $m\in \{1,\ldots,n\}\setminus\{i\}$, we deduce that $\langle u_i,\tilde w_i\rangle=0$, which contradicts \eqref{remref3}. 
 It follows that $u_1,\ldots,u_n$ are linearly independent.

We define $w_1=u_1$, and for $i=2,\ldots,n$, we let $w_i$ be
the unit vector in ${\rm lin}\,\{u_1,\ldots,u_i\}$
which is orthogonal to $u_1,\ldots,u_{i-1}$ and satisfies
$\langle w_i,u_i\rangle>0$.
As $w_1,\dots,w_n$ form an orthonormal basis, ${\rm lin}\,\{u_1,\ldots,u_{m}\}={\rm lin}\,\{w_1,\ldots,w_{m}\}$ and $\langle \tilde{w}_i,u_j\rangle=0$ for $j\leq i-1$, it follows that $u_i=\sum_{j=1}^i\beta_jw_j$ and
$\tilde{w}_i=\sum_{m=i}^n\gamma_mw_m$
for $\beta_j,\gamma_m\in\R$ where $\langle w_i,u_i\rangle>0$
and $\langle \tilde{w}_i,u_i\rangle\geq 0$
yield $\beta_i,\gamma_i\in[0,1]$. Thus
$\langle u_i,w_i\rangle=\beta_i\ge \beta_i\gamma_i=\langle u_i,\tilde w_i\rangle$ or in other words
$\angle(w_i,u_i)\leq\angle(\tilde{w}_i,u_i)< 3\sqrt{n}\,\eta$, and hence if  $j\in{\cal V}_i$, then
$$
\angle(w_i,u_j)\leq \angle(w_i,u_i)+\angle(u_i,u_j)< 4\sqrt{n}\,\eta,
$$
which completes the argument.
\hfill\proofbox
\end{proof}

\medskip 

We also need an estimate about the Gaussian measure of cones.

\begin{lemma}
\label{ConeGaussianlower}
If $\alpha\in(0,\frac{\pi}2)$ and $w\in S^{n-1}$, then the convex cone
$C=\{x\in\R^n:\langle x,w\rangle \geq \|x\|\cos\alpha\}$ satisfies
$$
\gamma_n(C)\geq \frac{\sin^{n-1}(\alpha)}{\sqrt{2\pi n}}.
$$
\end{lemma}
\begin{proof} We write $\mathcal{H}^{n-1}(\cdot)$ to denote the $(n-1)$-dimensional Hausdorff measure normalized in such a way that it coincides with the $(n-1)$-dimensional Lebesgues measure on subsets of $w^\bot$. We observe that
the orthogonal projection of $C\cap S^{n-1}$ onto $w^\bot$ is an $(n-1)$-ball of radius $\sin \alpha$; therefore,
$$
\gamma_n(C)=\frac{\mathcal{H}^{n-1}(C\cap S^{n-1})}{\mathcal{H}^{n-1}(S^{n-1})}\ge 
\frac{ \sin^{n-1}(\alpha)\kappa_{n-1}}{n\kappa_n}.
$$
The logarithmic convexity of the Gamma function yields
$(\kappa_{n-1})^2 \geq \kappa_{n-2}\kappa_{n}$. Since  $\kappa_{n-2}\kappa_{n}^{-1}=\frac{n}{2\pi}$ 
we get 
$$\frac{\kappa_{n-1}}{\kappa_n}\geq \sqrt{\frac{n}{2\pi}},$$ from which the assertion follows. \hfill 
\proofbox
\end{proof}

\medskip 

Next we provide a quantitative estimate about polytopes whose vertices are close the vertices of a regular crosspolytope.

\begin{lemma}
\label{closeToRegCross}
If $p_1,\ldots,p_{2n}$ are the vertices of $B^n_1$, and $M={\rm conv}\{q_1,\ldots,q_k\}\subset\R^n$ with $\delta_H(\{p_1,\ldots,p_{2n}\},\{q_1,\ldots,q_k\})\leq\alpha$ for some $\alpha\in(0,1/\sqrt{n})$, then
$$
(1-\sqrt{n}\alpha)B^n_1\subset M\subset (1+\sqrt{n}\alpha)B^n_1.
$$
\end{lemma}
\begin{proof} We have $(B^n_1)^\circ=B^n_\infty\subset\sqrt{n} B^n_2$, hence $B^n_2\subset \sqrt{n} B^n_1$. By assumption, $B^n_1\subset M+\alpha B^n_2$. It follows that
$$
(1-\sqrt{n}\alpha)B^n_1+\sqrt{n}\alpha B^n_1\subset M+\sqrt{n}\alpha B^n_1,
$$
which yields the left inclusion. 

The right inclusion follows from $M\subset B^n_1+\alpha B^n_2\subset B^n_1+\alpha\sqrt{n}B^n_1$.
\hfill\proofbox
\end{proof}

\medskip 

The last auxiliary statement before Theorem~\ref{ell-sym-Loewner-stab} helps to estimate the variation of the $\ell$ function.

\begin{lemma}
\label{EllDifference}
If $R>0$ and $K\subset C\subset R B^n_2$ are $o$-symmetric  convex bodies, then
$$
\ell(K)-\ell(C)\geq  \left(\frac{n}{2\pi e}\right)^{\frac{n}2}R^{-(n+1)}V(C\setminus K).
$$
\end{lemma}
\begin{proof} It follows from \eqref{ellGaussianint}, and using substitutions like $x=ty$ and $t=s/R$ that
\begin{align*}
\ell(K)-\ell(C) &=\int_0^\infty(\gamma_n(tC)-\gamma_n(tK))\,dt=
\frac1{\sqrt{2\pi}^{n}}\int_0^\infty\int_{t(C\setminus K)}e^{-\frac{\|x\|^2}2}\,dx\,dt\\
&=\frac1{\sqrt{2\pi}^{n}}\int_0^\infty\int_{C\setminus K}t^ne^{-\frac{t^2\|y\|^2}2}\,dy\,dt\geq 
\frac{V(C\setminus K)}{\sqrt{2\pi}^{n}}\int_0^\infty t^ne^{-\frac{t^2R^2}2}\,dt\\
&=\frac{V(C\setminus K)}{R^{n+1}\sqrt{2\pi}^{n}}\int_0^\infty s^ne^{-\frac{s^2}2}\,ds
=\frac{V(C\setminus K)}{R^{n+1}\sqrt{2\pi}^{n}}2^{\frac{n-1}{2}}\Gamma\left(\frac{n+1}{2}\right).
\end{align*}
Since 
$$
\Gamma(x+1)\ge  \sqrt{2\pi x}\left(\frac{x}{e}\right)^{x},\quad x\ge 0,
$$
we get
$$
\frac{1}{\sqrt{2\pi}^{n}}2^{\frac{n-1}{2}}\Gamma\left(\frac{n+1}{2}\right)
\ge \sqrt{e\pi}\left(1-\frac{1}{n}\right)^{\frac{n}{2}}\left(\frac{n}{2\pi e}\right)^{\frac{n}{2}}\ge \left(\frac{n}{2\pi e}\right)^{\frac{n}{2}} ,\quad n\ge 2,
$$
which proves the assertion.
\hfill\proofbox
\end{proof}

\medskip 

Now we are ready to prove a stability version of Theorem~\ref{ell-sym-Loewner}. 

\begin{theo}
\label{ell-sym-Loewner-stab}
There exists some explicitly calculable absolute constant $c\geq 1$ with the following properties. If $K$ is an origin symmetric convex body in $\R^n$ such that $B^n_2\supset K$ is the L\"owner ellipsoid of $K$, and $\ell(K)\geq (1-\varepsilon)\ell(B^n_1)$ for some $\varepsilon\in(0,\varepsilon_0)$ for $\varepsilon_0=n^{-cn^2}$, then there exists an orthogonal transformation $\Phi\in O(n)$ such that $\delta_{\rm H}(K,\Phi B^n_1)\leq n^{cn^2}\cdot \varepsilon$.
\end{theo}
\begin{proof}
According to John's  characteristic condition,
there exist $u_1,\ldots,u_k\in S^{n-1}\cap \partial K$ and $c_1,\ldots,c_k>0$ with $n\leq k \leq \frac{n(n+1)}2$ that
satisfy \eqref{John-matrix}.

We note that up to \eqref{Ccup-C}, we do not use the condition $\ell(K)\geq(1-\varepsilon) \ell(B^n_1)$, only 
that \eqref{John-matrix} holds and $k\leq n^2$.
Applying Lemma~\ref{ciuibig} with $v_i=\sqrt{c_i}\,u_i$,
and using that $\binom{k}{n}\leq \binom{n^2}{n}\leq (\frac{e\cdot n^2}{n})^n=(en)^n$,
we may assume after re-indexing that
$$
c_1\cdots c_n\cdot\Big|\det[u_1,\ldots,u_n]\Big|^2\geq (en)^{-n}.
$$
We have $|\det[u_1,\ldots,u_n]|\leq 1$ as $u_1,\ldots,u_n$ are unit vectors, and $c_i\leq 1$ for $i=1,\ldots,n$ by
\eqref{John-ci-upper}; therefore,
\begin{equation}
\label{cidetlower}
c_i \geq (en)^{-n}\mbox{ \ for $i=1,\ldots,n$, and \ }
\Big|\det[u_1,\ldots,u_n]\Big|\geq (en)^{-n/2}.
\end{equation}

We define $M={\rm conv}\{\pm u_1,\ldots,\pm u_k\}$. The core claim of our argument for proving
Theorem~\ref{ell-sym-Loewner-stab} is that  there exist  $\Phi\in O(n)$ and $\gamma>0$ depending on $n$ such that
\begin{equation}
\label{Mclosetocrosspolytope}
(1- \gamma\varepsilon) \Phi B_1^n\subset M\subset K .
\end{equation}

If $k=n$, then $u_1,\ldots,u_n$ form an orthonormal basis, and hence 
$\Phi B_1^n=M$ for some $\Phi\in O(n)$, and thus  \eqref{Mclosetocrosspolytope} readily holds.
Therefore, we assume that $k>n$.

Let $\tilde{\eta}\in[0, \frac{\pi}2)$ be minimal with the  property that
for any $j=n+1,\ldots,k$ there exists
$i\in\{1,\ldots,n\}$ satisfying
$|\langle u_i,u_j\rangle|\geq \cos\tilde{\eta}$ (here $\tilde{\eta}<\frac{\pi}2$ follows from \eqref{cidetlower}), and let
\begin{equation}
\label{etatileeta}
 \eta=\min\left\{\tilde{\eta},\frac1{3n\cdot (en)^{n/2}}\right\} .
\end{equation}
In particular, we may assume that
\begin{equation}
\label{ukfar}
|\langle u_i,u_k\rangle|\leq \cos\tilde{\eta}\mbox{ \ \ for $i=1,\ldots,n$.}
\end{equation}
Possibly replacing $u_i$ by $-u_i$ if $i=1,\ldots,n$
, we may assume that
\begin{equation}
\label{u1nklambda}
u_k=\sum_{i=1}^n\lambda_i u_i\mbox{ \ \ for $\lambda_1,\ldots,\lambda_n\geq 0$
where at least two $\lambda_i>0$.}
\end{equation}
 We observe that
$$
v=\frac{1}{\lambda_1+\cdots+\lambda_n}\cdot  u_k=\sum_{i=1}^n\frac{\lambda_i}{\lambda_1+\cdots+\lambda_n}\cdot  u_i\in
({\rm int}\,B^n_2)\cap {\rm conv}\{u_1,\ldots,u_n\},
$$
and hence $\|v\|<1$. We claim that
\begin{equation}
\label{vnormupper}
\lambda_1+\cdots+\lambda_n=\frac1{\|v\|}\geq 1+\frac1{3\cdot (en)^{n/2}}\cdot  \eta.
\end{equation}
For $x\in \R^n$ and a linear subspace $L\subset \R^n$, we write ${\rm dist}(x,L)$ for the Euclidean distance of $x$ from $L$.
In addition, for $i=1,\ldots,n$, we write $u_1,\ldots,\check{u}_i,\ldots,u_n$ to denote the list of $n-1$ vectors where $u_i$ is removed from the list $u_1,\ldots,u_n$. In particular, for $i=1,\ldots,n$, \eqref{cidetlower} yields that the linear $(n-1)$-space $L_i={\rm lin}\{u_1,\ldots,\check{u}_i,\ldots,u_n\}$
satisfies
\begin{equation}
\label{distuiface}
{\rm dist}(u_i,L_i)=\frac{\Big|\det[u_1,\ldots,u_n]\Big|}{{\Big|\det_{n-1}[u_1,\ldots,\check{u}_i,\ldots,u_n]\Big|}}
\geq (en)^{-n/2}.
\end{equation}

We may assume that $\lambda_1\geq \lambda_2\geq \ldots\geq \lambda_n$, and hence $\lambda_2>0$ and
\begin{equation}
\label{distvfacerelative}
{\rm dist}(v,L_1)\geq \frac1n\cdot {\rm dist}(u_1,L_1)
\geq \frac1{n\cdot (en)^{n/2}}.
\end{equation}

Let $v'\in S^{n-1}$ be the other endpoint of the secant of $B^n_2$ such that $v\in{\rm conv}\{u_1,v'\}$,
and hence $v'\neq -u_1$ by $\lambda_2>0$, and
 there exists a $v''\in L_1\cap {\rm conv}\{v,v'\}$ with $v''\neq o$
 as $v'\neq -u_1$. 
It follows from \eqref{distvfacerelative} that
the distance of $v$ from the line of $o$ and $v''$ is
at least $\frac1{n\cdot (en)^{n/2}}$, and hence 
$\|v\|\leq 1$ yields that
$\sin\angle (v,v'')\geq \frac1{n\cdot (en)^{n/2}}$.
We deduce that
\begin{align}
\label{anglevvprime}
\angle (v,v')&\geq 
\angle (v,v'')\geq \sin\angle (v,v'')\geq \frac1{n\cdot (en)^{n/2}}\geq \eta\\
\label{angleu1ovprime}
\cos \angle (v',u_1,o)&=\frac{\|u_1-v'\|}2\geq \frac{{\rm dist}(u_1,L_1)}2\geq
\frac1{2\cdot (en)^{n/2}}.
\end{align}
On the other hand, \eqref{ukfar} yields that
\begin{equation}
\label{anglevu1}
\angle (v,u_1)=\angle (u_k,u_1)\geq \tilde{\eta}\geq \eta.
\end{equation}
It follows from \eqref{anglevvprime} and \eqref{anglevu1} that there exists $v_0\in {\rm conv}\{v,u_1\}$ with $\angle (u_1,v_0)=\eta$, and it satisfies
$\|v\|\leq \|v_0\|$ and $\angle (v_0,u_1,o)+\eta\leq \frac{\pi}2$; therefore,
applying the Law of Sine in the triangle ${\rm conv}\{v_0,u_1,0\}$ implies
\begin{align*}
\lambda_1+\cdots+\lambda_n=\frac1{\|v\|}&\geq \frac1{\|v_0\|}=\frac{\|u_1-o\|}{\|v_0-o\|}=
\frac{\sin\angle (u_1,v_0,o)}{\sin\angle (v_0,u_1,o)}\\
&=
\frac{\sin(\angle (v',u_1,o)+\eta)}{\sin\angle (v',u_1,o)}
=\cos \eta+\frac{\cos\angle (v',u_1,o)}{\sin\angle (v',u_1,o)}\cdot
\sin\eta,
\end{align*}
and hence \eqref{angleu1ovprime} yields that
$$
\lambda_1+\cdots+\lambda_n\geq \cos \eta+\frac1{2\cdot (en)^{n/2}}\cdot
\sin\eta\geq 1+\frac1{3 \cdot  (en)^{n/2}}\cdot \eta
$$
where we used that $\eta\leq \frac1{3n\cdot (en)^{n/2}}\le \frac{1}{6}c$ with $c=(en)^{-n/2}$,
and that by basic calculus we have $\frac{\partial}{\partial t}(\cos t+\frac{\bar c}{2}\cdot \sin t)\geq \frac{\bar c}3$
whenever $\bar c\in(0,1]$ and $t\in(0,\frac{\bar c}6]$. In turn, we conclude the claim \eqref{vnormupper}.

Next let the $(n-1)$-ball $B^n_2\cap {\rm aff}\{u_1,\ldots,u_n\}$
have center $w_0$ and radius $\varrho$, and
let $w=w_0/\|w_0\|\in S^{n-1}$. Since the $(n-1)$-dimensional volume of
${\rm conv}\{u_1,\ldots,u_n\}$ is at most the $(n-1)$-dimensional volume
$(\frac{n}{n-1})^{\frac{n-1}2}\frac{\sqrt{n}}{(n-1)!}$
of the regular $(n-1)$-dimensional simplex of circumradius $1$,
and the distance of $o$ from ${\rm aff}\{u_1,\ldots,u_n\}$ is $\|w_0\|$,
we deduce from \eqref{cidetlower} that
\begin{align*}
\frac1n\cdot \|w_0\|\cdot\sqrt{e}\cdot\frac{\sqrt{n}}{(n-1)!}&>
\frac1n\cdot \|w_0\|\cdot\left(\frac{n}{n-1}\right)^{\frac{n-1}2}\frac{\sqrt{n}}{(n-1)!}\\
&>
V({\rm conv}\{o,u_1,\ldots,u_n\})>\frac1{n!}\cdot (en)^{-n/2}.
\end{align*}
It follows that
$$
 \|w_0\|>(en)^{-\frac{n+1}{2}}.
$$
Writing $\beta= \angle(w,u_i)\in(0,\frac{\pi}2)$ for $i=1,\ldots,n$, we thus get
$$
\sin\left(\frac{\pi}2-\beta\right)=\cos\beta=\|w_0\|> (en)^{-\frac{n+1}{2}},
$$
which in turn yields that
\begin{equation}
\label{gammaup}
\beta<\frac{\pi}2-\frac1{(en)^{\frac{n+1}{2}}}.
\end{equation}
Let $C$ be the cone of vectors whose angle with $w$ is at most $\frac1{2(en)^{\frac{n+1}{2}}}$; namely,
$$
C=\left\{x\in\R^n:\langle x,w\rangle\geq  \|x\|\cdot \cos \frac1{2(en)^{\frac{n+1}{2}}}\right\}.
$$
We deduce from
\eqref{gammaup} that if $x\in C\setminus  \{o\}$ and $i=1,\ldots,n$, then
$$
\angle (x,u_i)\leq \frac{\pi}2-\frac1{2(en)^{\frac{n+1}{2}}},
$$
and hence \eqref{cidetlower} implies that
\begin{equation}
\label{xuilow}
c_i\langle x,u_i\rangle\geq \frac{\|x\|}{(en)^{n}}\cdot \sin\frac1{2(en)^{\frac{n+1}{2}}}>
 \frac{\|x\|}{(en)^{n}}\cdot \frac1{4(en)^{\frac{n+1}{2}}}>\frac{\|x\|}{4(en)^{\frac{3n+1}2}}.
\end{equation}

We also need an upper bound on $\lambda_i$ for $i=1,\ldots,n$. Let $u_1^*,\ldots,u_n^*$ be the
dual basis to $u_1,\ldots,u_n$; namely, $\langle u_i^*,u_i\rangle=1$ for $i=1,\ldots,n$,
and $\langle u_i^*,u_j\rangle=0$ if $i\neq j$. It follows that if $i=1,\ldots,n$, then
\eqref{distuiface} yields that
$$
\|u_i^*\|=\frac{{\Big|\det_{n-1}[u_1,\ldots,\check{u}_i,\ldots,u_n]\Big|}}{\Big|\det[u_1,\ldots,u_n]\Big|}
\leq (en)^{n/2};
$$
therefore,
\begin{equation}
\label{lambdaiup}
\lambda_i=\langle u_i^*,u_k\rangle\leq \|u_i^*\|\leq (en)^{n/2}.
\end{equation}

We consider
$$
\theta=\frac1{4(en)^{2n+1}}.
$$
It follows from \eqref{xuilow} and \eqref{lambdaiup} that if $x\in C$ and $i=1,\ldots,n$, then
\begin{equation}
\label{ciminuslambdai}
c_i\langle x,u_i\rangle-\theta\cdot\lambda_i\|x\|\geq 0.
\end{equation}
Now  \eqref{John-x-1} and \eqref{u1nklambda} imply that if $x\in C$, then
$$
x=\left(c_k\langle x,u_k\rangle+\theta\cdot\|x\|\right)u_k+
\sum_{i=1}^n\left(c_i\langle x,u_i\rangle-\theta\cdot\lambda_i\|x\|\right)u_i+
\sum_{n<i<k}c_i\langle x,u_i\rangle u_i,
$$
where the last term is void if $k=n+1$, and hence applying first \eqref{ciminuslambdai} and later
\eqref{vnormupper} we get
\begin{align*}
\left(\sum_{i=1}^kc_i|\langle x,u_i\rangle|\right)-\|x\|_M&=\left(\sum_{i=1}^kc_i|\langle x,u_i\rangle|\right)-
\inf\left\{\sum_{i=1}^k|\alpha_i|:\,x=\sum_{i=1}^k\alpha_iu_i\right\}\\
&\geq
-\theta\cdot\|x\|+\sum_{i=1}^n\theta\cdot\lambda_i\|x\|\geq
\theta\cdot\|x\|\cdot \frac1{3 (en)^{n/2}}\cdot  \eta\\
&
\geq
 \frac{\|x\|}{12(en)^{3n}}\cdot  \eta.
\end{align*}
We observe that if $x\in C\cup(-C)$ and $y\not\in C\cup(-C)$, then
$|\langle w,x\rangle|\geq |\langle w,y\rangle|$, and hence
\begin{align}
\label{Ccup-C}
\int_{C\cup(-C)}\|x\|\,\gamma_n(\dint x)&\geq
\int_{C\cup(-C)}|\langle w,x\rangle|\,\gamma_n(\dint x)\nonumber\\
&\geq 2
\gamma_n(C)
\int_{\R^n}|\langle w,x\rangle| \,\gamma_n(\dint x).
\end{align}

Now we start to use the condition  $\ell(K)\geq(1-\varepsilon) \ell(B^n_1)$.
It follows from \eqref{xnormKM} that $\|x\|_M\leq \sum_{i=1}^kc_i|\langle x,u_i\rangle|$ for all $x\in\R^n$; therefore,
\eqref{ellB1uici},  Lemma~\ref{ConeGaussianlower} and $\ell(M)\geq \ell(K)\geq(1-\varepsilon) \ell(B^n_1)$ yield
\begin{align*}
\varepsilon \ell(B^n_1)&\geq  \ell(B^n_1)-\ell(M)=
\int_{\R^n}\sum_{i=1}^kc_i|\langle x,u_i\rangle|\,\gamma_n(\dint x)-
\int_{\R^n}\|x\|_M\,\gamma_n(\dint x)\\
&\geq
\int_{C\cup(-C)}\frac{\|x\|}{12(en)^{3n}}\cdot  \eta \,\gamma_n(\dint x)\geq
\frac{\gamma_n(C)}{6(en)^{3n}}\int_{\R^n}|\langle w,x\rangle| \,\gamma_n(\dint x)\cdot  \eta\\
&=\frac{\gamma_n(C)}{6n(en)^{3n}}\cdot  \ell(B^n_1)\cdot  \eta
\geq \left(\sin\frac1{2(en)^{\frac{n+1}{2}}}\right)^{n-1}\cdot
\frac{1}{18 n^2(en)^{3n}}\cdot  \ell(B^n_1)\cdot  \eta.
\end{align*}

It follows that there exists some absolute constant $c(1)\geq 1$ such that
$$
\eta\leq n^{c(1)n^2}\varepsilon.
$$
We choose an absolute constant $c(2)\geq 4c(1)$ such that $n^{(c(1)-c(2))n^2}<\frac1{3n\cdot (en)^{n/2}}$ for $n\geq 2$, thus
\[ 
n^{c(1)n^2}\varepsilon<\frac1{3n\cdot (en)^{n/2}} \mbox{ in \eqref{etatileeta} if }0<\varepsilon<n^{-c(2)n^2}.
\]
Therefore, if $\varepsilon_0=n^{-c(2)n^2}$, then
$$
\eta=\tilde{\eta}.
$$
In particular,  for any $j=n+1,\ldots,k$ there exists
$i\in\{1,\ldots,n\}$ satisfying
$|\langle u_i,u_j\rangle|\geq \cos\eta$, and hence
 Lemma~\ref{almostort0} yields the existence
of an orthonormal basis $w_1,\ldots,w_n$
such that
\begin{equation}\label{eq:sec2close}
\delta_{\rm H}\left(\{\pm w_1,\ldots,\pm w_n\},\{\pm u_1,\ldots,\pm u_k\}\right)<4\sqrt{n}\,\eta<n^{2c(1)n^2}\varepsilon.
\end{equation}
In turn, we conclude the claim \eqref{Mclosetocrosspolytope} by Lemma~\ref{closeToRegCross}; namely,
$$
(1- \gamma\varepsilon) B_1^n\subset M
$$
for $\gamma=n^{3c(1)n^2}$ as $n^{3c(1)n^2}\varepsilon<n^{(3c(1)-c(2))n^2}\le n^{-n^2}<n^{-1}$ by $c(2)\geq 4c(1)$ and $c(1)\geq 1$.
Since $M\subset K$, we deduce that $(1- \gamma\varepsilon) B_1^n\subset K$.

Now if $z\in K\subset B^n_2$ and $z\notin (1-\gamma\varepsilon)B^n_1$, then $\|z\|_{B_1^n}=
(1+t)(1-\gamma\varepsilon)$ for some $t>0$.  There exists a facet $F$ of $(1-\gamma\varepsilon)B^n_1$ such that $z_0=\frac1{1+t}\,z\in F$. Hence the polytope
$$
P_{z}={\rm conv}\left(\{\pm z\}\cup (1- \gamma\varepsilon) B_1^n\right)
$$
satisfies 
\begin{align*}
V(P_z)-V\left((1- \gamma\varepsilon) B_1^n\right)
&\geq 2t\cdot V\left({\rm conv}\left\{o,F\right\}\right)
=2t(1- \gamma\varepsilon)^n\frac{V(B^n_1)}{2^n}
\\
&>\nu_1(n)\ell(B^n_1)\cdot t
\end{align*}
with $\nu_1(n)=(1-n^{-1})^n(n!\ell(B^n_1))^{-1}\ge (4n\cdot n!)^{-1}$, since $\ell(B^n_1)\le n$. 
As $P_z\subset B^n_2$, Lemma~\ref{EllDifference} yields
that
$$
\ell\left((1- \gamma\varepsilon) B_1^n\right)-\ell(P_z)\geq \left(\frac{n}{2\pi e}\right)^{\frac{n}2}\nu_1(n)\ell(B^n_1)\cdot t\geq n^{-4n}\ell(B^n_1)\cdot t.
$$
Here we used that $n!\le \sqrt{2\pi n}\left(\frac{n}{e}\right)^n\cdot 1.2$ for $n\ge 2$ and hence
\begin{equation}\label{eq:2bound}
\left(\frac{2\pi e}{n}\right)^{\frac{n}{2}}4n\cdot n!\le n^{4n}.
\end{equation}
We conclude that
\begin{align*}
(1-\varepsilon)\ell(B^n_1)&\leq \ell(K)\leq \ell(P_z)\leq \ell\Big((1- \gamma\varepsilon) B_1^n\Big)-n^{-4n}t\ell(B_1^n)\\
&=
\Big((1- \gamma\varepsilon)^{-1}-n^{-4n}t\Big)\ell(B_1^n)
\leq
(1+2\gamma\varepsilon-n^{-4n}t)\ell(B_1^n),
\end{align*}
and therefore $t\leq (2\gamma+1)n^{4n} \cdot \varepsilon\leq n^{c(3)n^2} \cdot \varepsilon$ with 
 an absolute constant  $c(3)>0$. If we set $ \widehat{\gamma}=n^{c(3)n^2}$, then
$$
K\subset (1+ \widehat{\gamma}\varepsilon)B_1^n.
$$
Choosing $c=\max\{c(2),c(3)\}$, the assertion follows.
\hfill\proofbox
\end{proof}

\section{Extremal $\ell$-norm when the John ellipsoid is a ball and the rank one Brascamp--Lieb inequality}
\label{secBrascamp-Lieb}

One of our main tools in this section is Barthe's proof of the rank one geometric Brascamp--Lieb inequality by means of optimal transport of probability measures  (see \cite{Bar97,Bar98}). The geometric form of the corresponding general analytic inequality, verified originally
by Brascamp, Lieb \cite{BrL76},  was identified by   Ball \cite{Bal89} as an important case
that is perfectly suited for various geometric applications. 
For a more detailed discussion of the  Brascamp--Lieb inequality (including higher ranks), we refer to 
Carlen, Cordero-Erausquin \cite{CCE09}, Lieb \cite{Lie90},   Barthe \cite{Bar98}, 
  Valdimarsson \cite{Val08}, 
Bennett, Carbery, Christ, Tao \cite{BCCT08}, and Barthe, Cordero-Erausquin, Ledoux,
Maurey  \cite{BCLM11}.

To set up the form \eqref{BL} of the Brascamp--Lieb inequality, 
let the unit vectors $u_1,\ldots,u_k\in S^{n-1}$  and $c_1,\ldots,c_k>0$ satisfy $u_i\neq \pm u_j$ for $i\neq j$ and \eqref{John-matrix}, i.e., 
\begin{equation}
\label{ciui}
\sum_{i=1}^kc_i u_i\otimes u_i=\Id_n.
\end{equation}
During the argument in Barthe \cite{Bar98}, the following four consequences of \eqref{ciui} observed by  Ball \cite{Bal89}
(see also \cite{Bar98} for a simpler proof of \eqref{BallBarthe}) play crucial roles:
If $k\geq n$, $c_1,\ldots,c_k>0$ and  $u_1,\ldots,u_k\in S^{n-1}$ satisfy
\eqref{ciui}, then we have the following properties:
\begin{description}
\item[Ball--Barthe inequality:] For any $t_1,\ldots,t_k>0$, we have
\begin{equation}
\label{BallBarthe}
\det \left(\sum_{i=1}^kt_ic_i u_i\otimes u_i\right)\geq \prod_{i=1}^k t_i^{c_i}.
\end{equation}

\item[Quadratic inequality: ]
For $z=\sum_{i=1}^kc_i\theta_i u_i$ with
$\theta_1,\ldots,\theta_k\in\R$, we have
\begin{equation}
\label{BLRBLquad}
\|z\|^2=\sum_{i=1}^kc_i\langle z,u_i\rangle^2\le \sum_{i=1}^kc_i\theta_i^2.
\end{equation}

\item[Properties of $c_1,\ldots,c_k$:]
\begin{align}
\label{ciatmost10}
c_i&\leq 1\mbox{ \ for $i=1,\ldots,k$,}\\
\label{cisum0}
c_1+\cdots+c_k&=n.
\end{align}

\end{description}

We only need the Brascamp--Lieb inequality in the following special case:

\begin{theo}\label{thm:BLspecialcase}
Let $b>0$. Let $f$ be a non-negative log-concave function on $\R$ such that $f(t)>0$ if and only if $|t|\leq b$. Let $u_1,\ldots,u_k\in S^{n-1}$  and $c_1,\ldots,c_k>0$ satisfy  $u_i\neq \pm u_j$ for $i\neq j$ and \eqref{ciui}. Then
\begin{equation}
\label{BL}
\int_{\R^n}\prod_{i=1}^kf(\langle x,u_i\rangle)^{c_i}\, \dint x \leq
\prod_{i=1}^k\left(\int_{\R}f(t)\, \dint t\right)^{c_i}.
\end{equation}
\end{theo}
\noindent{\bf Remark}
It was proved by Barthe \cite{Bar98} that equality in \eqref{BL} holds if and only if  $k=n$ and $u_1,\ldots,u_n$ form an orthonormal basis of $\R^n$. We carry out the proof of Theorem \ref{thm:BLspecialcase} since we will refer to the argument in our proof of Proposition \ref{gammanbP-stab}.

\medskip 

\begin{proof}
For the proof of \eqref{BL},  by scaling we may assume that 
$$\int_{\R}f(t)\dint t=\int_{-b}^b f(t)\, \dint t=1.$$
We follow Barthe in using a transport of measure argument. We write $\gamma_1(t)=\sqrt{{2\pi}}^{-1}\,e^{-t^2/2}$, $t\in\R$, for the standard one-dimensional Gaussian density  (there is no danger of confusing it with the corresponding measure although we use the same symbol). 
Let $\varphi:(-b,b)\to\R$  be the transport map determined by the property
$$
\int_{-\infty}^x f(t)\, \dint t=\int_{-\infty}^{\varphi(x)} \gamma_1(t)\, \dint t
$$
for $x\in(-b,b)$.
Here $\varphi$ is $C^1$ as $f$ is continuous on $(-b,b)$, and 
\begin{equation}
\label{masstrans}
f(x)=\gamma_1(\varphi(x))\cdot \varphi'(x).
\end{equation}
For
$$
{\cal C}=\{x\in\R^n:\, \langle u_i,x\rangle\in (-b,b)\;\text{ for } i=1,\ldots,k\},
$$
we consider the transformation $\Theta:{\cal C}\to\R^n$ with
$$
\Theta(x)=\sum_{i=1}^kc_i\, \varphi(\langle u_i,x\rangle )\,u_i,\qquad x\in {\cal C},
$$
which satisfies
$$
d\Theta(x)=\sum_{i=1}^kc_i\, \varphi'(\langle u_i,x\rangle )\,u_i\otimes u_i.
$$
It is {known} that $d\Theta$ is positive definite and $\Theta:{\cal C}\to\R^n$ is
 injective (see \cite{Bar97,Bar98}).
Therefore, using first \eqref{masstrans},  then the Ball--Barthe inequality \eqref{BallBarthe} with $t_i=\varphi'(\langle u_i,x\rangle)$,
and then the definition of $\Theta$ and  \eqref{BLRBLquad}, the following argument leads to the Brascamp--Lieb inequality
in this special case:
\begin{align}
&\int_{\R^n}\prod_{i=1}^kf(\langle u_i,x\rangle)^{c_i}\,\dint x=
\int_{{\cal C}}\prod_{i=1}^kf(\langle u_i,x\rangle)^{c_i}\,\dint x\nonumber\\
\label{BLstep1}
&=
\int_{{\cal C}}\left(\prod_{i=1}^k\gamma_1(\varphi(\langle u_i,x\rangle))^{c_i}\right)
\left(\prod_{i=1}^k\varphi'(\langle u_i,x\rangle)^{c_i}\right)\,\dint x\\
\nonumber
&\leq \frac{1}{(2\pi)^{\frac{n}2}}\int_{{\cal C}}\left(\prod_{i=1}^ke^{-c_i\varphi(\langle u_i,x\rangle)^2/2}\right)
\det\left(\sum_{i=1}^kc_i\varphi'(\langle u_i,x\rangle )\,u_i\otimes u_i\right)\,\dint x\\
\nonumber
&\leq  \frac{1}{(2\pi)^{\frac{n}2}}\int_{{\cal C}}e^{-\|\Theta(x)\|^2/2}\det\left( d\Theta(x)\right)\,\dint x\\
\label{BLstep3}
&\leq \frac{1}{(2\pi)^{\frac{n}2}} \int_{\R^n}e^{-\|y\|^2/2}\,\dint y=1,
\end{align}
which proves the assertion.
\hfill\proofbox
\end{proof}

\medskip 

We now use the Brascamp--Lieb inequality to prove the following estimate:

\begin{prop}
\label{gammanbP-prop}
Let $b>0$. For $u_1,\ldots,u_k\in S^{n-1}$  and $c_1,\ldots,c_k>0$ satisfying \eqref{ciui}, the $o$-symmetric polytope
$P=\{x\in\R^n:\,|\langle x,u_i\rangle|\leq 1,\;i=1,\ldots,k\}$ satisfies
\begin{equation}
\label{gammabP}
\gamma_n(bP)\leq \gamma_n (bB^n_\infty).
\end{equation}
Equality holds if and only if $k=n$, $u_1,\ldots,u_n$ is an orthonormal basis and hence $P$ is a rotated copy of $B^n_\infty$.
\end{prop}
\begin{proof}
Let
$$
f_{(b)}(t)=
\left\{
\begin{array}{cl}
\frac1{\sqrt{2\pi}}\,e^{-\frac{t^2}2},&\mbox{ if $|t|\leq b$},\\[1ex]
0,&\mbox{ if $|t|> b$.}
\end{array}
\right.
$$
It follows that
\begin{align}
\label{gammabPright}
\gamma_n ([-b,b]^n)&=\left(\int_{\R}f_{(b)}(t)\,\dint t\right)^n=
\prod_{i=1}^k\left(\int_{\R}f_{(b)}(t)\, \dint t\right)^{c_i},
\end{align}
and
\begin{align}
&\gamma_n(bP)=\frac{1}{\sqrt{2\pi}^n}\int_{bP}e^{-\frac{1}{2}\|x\|^2}\, \dint x\nonumber\\
&=\int_{\R^n}\left(\frac{1}{\sqrt{2\pi}}\right)^{\sum_{i=1}^kc_i}\prod_{i=1}^k\mathbf{1}\{x\in\R^n:|\langle x,u_i\rangle|\le b\}e^{-\frac{1}{2} \sum_{i=1}^kc_i\langle x,u_i\rangle^2}\,\dint x\nonumber\\
\label{gammabPleft}
&=\int_{\R^n}\prod_{i=1}^kf_{(b)}(\langle x,u_i\rangle)^{c_i}\,\dint x,
\end{align}
and hence the Brascamp--Lieb inequality \eqref{BL} yields \eqref{gammabP}.

If equality holds, then equality must holds in \eqref{BL} for the function $f_{(b)}$. By the remark after Theorem \ref{thm:BLspecialcase}, the assertion about the equality case follows. 
\hfill\proofbox
\end{proof}

\medskip

We recall that, according to \eqref{ellGaussianint}, we have
\begin{equation}
\label{ellGaussianint0}
\ell(K)=\int_0^\infty(1-\gamma_n(tK))\,\dint t.
\end{equation}

\begin{theo}[Schechtman, Schmuckenschl\"ager]
\label{ell-sym-John}
If  $K$ is an $o$-symmetric convex body in $\R^n$ such that $B^n_2\subset K$ is the John ellipsoid of $K$, then $\ell(K)\geq \ell(B^n_\infty)$. Equality holds if and only if $K$ is a cube circumscribed around $B^n_2$. 
\end{theo}
\begin{proof}
Let $u_1,\ldots,u_k\in S^{n-1}$ and $c_1,\ldots,c_k>0$ such that \eqref{ciui} holds. 
If 
$$
P=\{x\in\R^n:\,|\langle x,u_i\rangle|\leq 1,\;i=1,\ldots,k\},
$$
then $K\subset P$ and $\ell(K)\geq \ell(P)$.
On the other hand, an application of the estimate \eqref{gammabP} in the formula \eqref{ellGaussianint0}
 yields that
$\ell(P)\geq \ell (B^n_\infty)$. 

If equality holds, then clearly $K=P$ and $P$ is a rotated copy of $B^n_\infty$ by the equality case in Proposition \ref{gammanbP-prop}.
\hfill\proofbox
\end{proof}

\section{Auxiliary statements to strengthen Theorem~\ref{ell-sym-John}}
\label{secAuxiliaryJohn}

In order to obtain a stability version of Theorem~\ref{ell-sym-John}, let
$$
\Gamma_b=\frac1{\sqrt{2\pi}}\int_{-b}^be^{-\frac{s^2}2}\,\dint s<1,
$$
for $b>0$, and consider the transport map $\varphi_{b}:(-b,b)\to\R$ such that
\begin{equation}
\label{phibdef}
\frac1{\sqrt{2\pi}}\int_{-\infty}^{\varphi_{b}(t)} e^{-\frac{s^2}2}\,\dint s
=\frac1{\Gamma_b\sqrt{2\pi}}\int_{-b}^t e^{-\frac{s^2}2}\,\dint s.
\end{equation}

Readily, $\varphi_{b}$ is an odd function, and hence $\varphi_b(0)=0$, $\varphi'_{b}(s)$ is even
and $\varphi''_{b}(s)$ is odd. Differentiation of \eqref{phibdef} yields
\begin{align}
\label{phibder}
 \varphi'_{b}(t)\cdot e^{-\frac{\varphi_{b}(t)^2}2}&= \frac{1}{\Gamma_b}\cdot e^{-\frac{t^2}2},\\
\label{phibderder}
 \varphi''_{b}(t)\cdot e^{-\frac{\varphi_{b}(t)^2}2}&=
\varphi_{b}(t)\cdot \varphi'_{b}(t)^2\cdot e^{-\frac{\varphi_{b}(t)^2}2}-\frac{t}{\Gamma_b}\cdot e^{-\frac{t^2}2}.
\end{align}

We collect various properties of the map $\varphi_b$ for later use.

\begin{lemma}
\label{propphib}
Let $b\in[1,2]$ and $t_b:=\left(2\log (1/\Gamma_b)\right)^{{1}/{2}}$. Then
\begin{enumerate}
\item[{\rm (a)}] $0.3\le t_b\le b$,
\item[{\rm (b)}] $1\le \varphi'_b(t)$ for $t\in [0,t_b]$ and $ \varphi'_b(t)\le 1.6$ for $t\in [0,0.3]$, 
\item[{\rm (c)}] $t\leq \varphi_b(t)\leq 1.6t$ for $t\in [0,0.3]$,
\item[{\rm (d)}] $\varphi''_{b}(t)\geq  
(1-\Gamma_b)\Gamma_b^{-2}\cdot t\ge 0.049\cdot t$ for $t\in [0,t_b]$.
\end{enumerate}
\end{lemma}

\begin{proof} (a) Since $b\mapsto \Gamma_b$ is increasing, we get $t_b\ge t_2\ge 0.3$. The assertion $t_b\le b$ is equivalent to 
$$
h(b):=\sqrt{\frac{2}{\pi}}\int_0^be^{-\frac{s^2}{2}}\, \dint s-e^{-\frac{b^2}{2}}\ge 0.
$$
Since $h(1)>0$ and $h$ is increasing, the assertion follows.

(b) Note that for $t\ge 0$ the condition $t\le t_b$ is equivalent to $ \frac{1}{\Gamma_b} e^{-\frac{t^2}{2}}\ge 1$. Hence, if $t\in [0,t_b]$ then 
$$
\varphi_b'(t)=\frac{1}{\Gamma_b}e^{-\frac{t^2}{2}}e^{\frac{\varphi_b(t)^2}{2}}\ge \frac{1}{\Gamma_b}e^{-\frac{t^2}{2}}\ge 1.
$$

To prove the upper bound for $\varphi_b'(t)$, first note that $t\mapsto \varphi_b(t)^2-t^2$ is non-decreasing for $t\in [0,t_b]$, since 
$$
\frac{d}{dt}\left(\varphi_b(t)^2-t^2\right)=2\varphi_b(t)\varphi_b'(t)-2t\ge 2t\cdot 1-2t\ge 0,
$$
and $b\mapsto \varphi_b(t)$ is non-increasing for $0\le t\le 0.3\le t_b\le b$. To see this, let $t$ be fixed. We show that the right side of \eqref{phibdef}, considered as a function of $b$, is non-increasing. Then the same is true for the left side of \eqref{phibdef}, hence $b\mapsto \varphi_b(t)$ must be non-increasing as well. For this, observe that
\begin{align*}
&\frac{\dint }{\dint b}\left(\int_{-b}^te^{-\frac{s^2}{2}}\, \dint s \cdot \left(\int_0^b e^{-\frac{s^2}{2}}\, \dint s\right)^{-1}\right)\\
&
=e^{-\frac{b^2}{2}}\left(\int_0^b e^{-\frac{s^2}{2}}\, \dint s\right)^{-2}\left(\int_{-b}^0e^{-\frac{s^2}{2}}\, \dint s -\int_{-b}^te^{-\frac{s^2}{2}}\, \dint s\right)\le 0,
\end{align*}
since $t\ge 0$. 

Using these facts, we get
\begin{align*}
\varphi_b'(t)&=\frac{1}{\Gamma_b}e^{\frac{1}{2}\left(\varphi_b(t)^2-t^2\right)}\le \frac{1}{\Gamma_1}e^{\frac{1}{2}\left(\varphi_b(0.3)^2-0.3^2\right)}\le \frac{1}{\Gamma_1}e^{\frac{1}{2}\left(\varphi_1(0.3)^2-0.3^2\right)}\le 1.548.
\end{align*}

(c) directly follows from (b) and $\varphi_b(0)=0$.

(d) Substituting $\varphi_b'(t)$ from \eqref{phibder} into \eqref{phibderder} and using that 
 $\varphi_b(t)\ge t$, we get 
\begin{align*}
\varphi''_{b}(t)&= 
\varphi_{b}(t) e^{{\varphi_{b}(t)^2}}\frac{1}{\Gamma_b^2}e^{-t^2}-\frac{t}{\Gamma_b} e^{-\frac{t^2}2} e^{\frac{\varphi_{b}(t)^2}2}\\
&\ge t\cdot \frac{1}{\Gamma_b}e^{\frac{1}{2}\left(\varphi_b(t)^2-t^2\right)}
\left(\frac{1}{\Gamma_b}e^{\frac{1}{2}\left(\varphi_b(t)^2-t^2\right)}-1\right)\\
&\ge \frac{1-\Gamma_b}{\Gamma_b^2}\cdot t\ge \frac{1-\Gamma_2}{\Gamma_2^2}\cdot t\ge 0.049\cdot t,
\end{align*}
which completes the argument.
\hfill \proofbox
\end{proof}

\medskip

We also need the following stability version of the Ball--Barthe inequality \eqref{BallBarthe} proved in   
B\"or\"oczky, Fodor, Hug \cite{BFH21}:

\begin{lemma}
\label{Ball-Barthe-stab}
If $k\geq n+1$, $t_1,\ldots,t_k>0$, $c_1,\ldots,c_k>0$ and  $u_1,\ldots,u_k\in S^{n-1}$ satisfy
(\ref{ciui}), and there exist $\beta>0$ and $n+1$ indices $ \{i_1,\ldots,i_{n+1}\}\subset \{1,\ldots,k\}$ such that
\begin{align*}
c_{i_1}\cdots c_{i_n}\det[u_{i_1},\ldots,u_{i_n}]^2&\geq \beta,\\
c_{i_2}\cdots c_{i_{n+1}}\det[u_{i_2},\ldots,u_{i_{n+1}}]^2&\geq \beta,
\end{align*}
then
$$
\det\left( \sum_{i=1}^kt_ic_iu_i\otimes u_i\right)\geq \left( 1+\frac{\beta(t_{i_1}-t_{i_{n+1}})^2}{4(t_{i_1}+t_{i_{n+1}})^2}\right) \prod_{i=1}^k t_i^{c_i}.
$$
\end{lemma}

 \section{Stability around the cube when the John ellipsoid is a ball }
 \label{secJohnBall}
 
To prove Proposition~\ref{gammanbP-stab}, we recall Lemma 3.2 in B\"or\"oczky, Hug \cite{BoH17}:

\begin{lemma}
\label{almostort-lemma}
Let  $v_1,\ldots,v_k\in \R^n\setminus \{0\}$
satisfy $\sum_{i=1}^k v_i \otimes v_i=\Id_n$, and let
$0<\eta<1/(3\sqrt{k})$. For any $i\in\{1,\ldots,k\}$, we assume  that $\|v_i\|\leq \eta$ or
there is some $j\in\{1,\ldots,n\}$ with $\angle(v_i,v_j)\leq \eta$.
Then there exists an orthonormal basis $w_1,\ldots,w_n$
such that $\angle(v_i,w_i)<3\sqrt{k}\,\eta$ for $i=1,\ldots,n$.
\end{lemma}

We will need the following consequence of Lemma \ref{almostort-lemma}. 

\begin{coro}
\label{almostort}
Let  $u_1,\ldots,u_{k}\in S^{n-1}$
and $c_1,\ldots,c_{k}\in S^{n-1}$
satisfy \eqref{ciui}, and let
$0<\eta<1/(3\sqrt{k})$. For any $i\in\{1,\ldots,k\}$, we assume that $c_i\leq \eta^2$ or
there is some $j\in\{1,\ldots,n\}$ with $|\langle u_i,u_j\rangle|\geq \cos\eta$.
Then there exists an orthonormal basis $w_1,\ldots,w_n$
such that $\angle(u_i,w_i)<3\sqrt{k}\,\eta$ for $i=1,\ldots,n$.
\end{coro}
\begin{proof} For $i=1,\ldots,k$ we define $v_i=\sqrt{c_i}u_i$. Then $v_1,\ldots,v_k\in \R^n\setminus \{0\}$
and $\sum_{i=1}^k v_i \otimes v_i=\Id_n$. For any $i\in\{1,\ldots,k\}$, by assumption we have $\|v_i\|=\sqrt{c_i}\le \eta$ or there is some $j\in\{1,\ldots,n\}$ such that $0\le \angle(v_i,v_j)\le \eta$ or $\pi\ge \angle(v_i,v_j)\ge \pi-\eta$. If $1\le i\le n$, then we can choose $j=i$ so that $\angle(v_i,v_j)=0\le \eta$. If $n+1\le i\le k$ and $\|v_i\|>\eta$, then there is some $j\in\{1,\ldots,n\}$ such that $0\le \angle(v_i,v_j)\le \eta$ or $0\le \angle(-v_i,v_j)\le \eta$. In the latter case, we simply replace $v_i$ by $-v_i$. The possibly modified sequence $v_1,\ldots,v_k$ satisfies all requirements of Lemma~\ref{almostort-lemma}. Since $v_1,\ldots,v_n$ remain unchanged, the assertion  follows by an application of Lemma~\ref{almostort-lemma}. 
\hfill\proofbox
\end{proof}

\medskip

The technical statements 
Lemma~\ref{absScalarProd} and
Lemma~\ref{deltaHvolCube} 
will be needed in the proof of Proposition~\ref{gammanbP-stab}:

\begin{lemma}
\label{absScalarProd}
Let $\eta\in(0,\frac{\pi}8)$, and let $u,u'\in S^1$ be such that $|\langle u,u'\rangle|\leq \cos\eta$. If $v\in S^1$ satisfies  $\cos \frac{3\pi}8\leq |\langle v,\frac{u-u'}{\|u-u'\|}\rangle|\leq \cos \frac{\pi}8$
and $\cos \frac{3\pi}8\leq |\langle v,\frac{u+u'}{\|u+u'\|}\rangle|\leq \cos \frac{\pi}8$, then
\begin{equation}
\label{absScalarProd-eq}
\Big|\,|\langle v,u\rangle|-|\langle v,u'\rangle|\,\Big|\geq  \eta/5.
\end{equation}
\end{lemma}
\begin{proof} Since \eqref{absScalarProd-eq} considers  $|\langle v,u\rangle|$ and $|\langle v,u'\rangle|$, we may assume, by symmetry, that $u=(\cos\alpha,\sin\alpha)^\top$ and $u'=(-\cos\alpha,\sin\alpha)^\top$, where $\frac{\eta}2\leq \alpha\leq\frac{\pi}2-\frac{\eta}2$, 
and hence $\frac{u+u'}{\|u+u'\|}=(0,1)$ and $\frac{u-u'}{\|u-u'\|}=(1,0)$,
and in addition, $v=(\cos\beta,\sin\beta)^\top$ where $\frac{\pi}8\leq |\beta|\leq  \frac{3\pi}8$. 

If $\beta\geq \frac{\pi}2-\alpha$,  then
\begin{align*}
|\langle v,u\rangle|-|\langle v,u'\rangle|&=\langle v,u\rangle-\langle v,u'\rangle
=2\cos\alpha\cos\beta\\
&\geq
2\cos\left (\frac\pi2-\frac\eta2\right )   \cos\frac{3\pi}8
=2 \sin \frac\eta2\sin\frac{\pi}8 \\
&
>2 \frac{2}{\pi}\frac\eta2\sin\frac{\pi}8>0.2436\cdot\eta>\frac\eta5.    
\end{align*}

If $\frac{\pi}8\leq \beta\leq \frac{\pi}2-\alpha$,  then
\begin{align*}
|\langle v,u\rangle|-|\langle v,u'\rangle|=\langle v,u\rangle+\langle v,u'\rangle=2\sin\alpha\sin\beta\geq
2\sin\frac{\eta}2   \sin\frac{\pi}8>\frac{\eta}5.
\end{align*}

If $\alpha-\frac{\pi}2\leq \beta\leq -\frac{\pi}8$,  then
\begin{align*}
|\langle v,u\rangle|-|\langle v,u'\rangle|&=\langle v,u\rangle+\langle v,u'\rangle=2\sin\alpha\sin\beta\leq
-2\sin\frac{\eta}2   \sin\frac{\pi}8<-\frac{\eta}5.    
\end{align*}

Finally, if $\beta\leq \alpha-\frac{\pi}2$,  then
$$
|\langle v,u\rangle|-|\langle v,u'\rangle|=-\langle v,u\rangle+\langle v,u'\rangle=-2\cos\alpha\cos\beta\leq
-2\sin\frac{\eta}{2}    \cos\frac{3\pi}8<-\frac{\eta}5,
$$
proving \eqref{absScalarProd-eq}.
\hfill \proofbox
\end{proof}

\medskip

For origin symmetric convex bodies $K,L\subset\R^n$, we consider the symmetric (volume) difference  
$$
\delta_{\rm vol}(K,L)=V(K\Delta L).
$$
In the following proposition, we first  provide local bounds in terms of the volume difference for the Gaussian measure.

\begin{prop}
\label{gammanbP-stab}
For $b\in[1,2]$, $n\leq k\leq\frac{n(n+1)}2$ and $u_1,\ldots,u_k\in S^{n-1}$  and $c_1,\ldots,c_k>0$ satisfying \eqref{ciui}, there exists $\Phi\in O(n)$ such that the polytope
$P=\{x\in\R^n:\,|\langle x,u_i\rangle|\leq 1,\;i=1,\ldots,k\}$ satisfies
\begin{align}
\label{gammabP-stabvol}
\gamma_n(bP)&\leq 
\gamma_n (bB^n_\infty)-2^{-80}n^{-40n}\delta_{\rm vol}(P,\Phi B^n_\infty)^4.
\end{align}
\end{prop}
\begin{proof} For $x\in P$, \eqref{ciui}  and
\eqref{cisum0} yield 
\begin{equation}
\label{PJonhHomothety}
\|x\|^2=\sum_{i=1}^kc_i\langle x,u_i\rangle^2\leq\sum_{i=1}^kc_i\leq n, \mbox{ \ and hence }P\subset \sqrt{n}\,B^n_2.
\end{equation}
In particular, for 
any $\Phi\in O(n)$, we have 
$\delta_{\rm vol}(P,\Phi B^n_\infty)\leq V(P)+V(\Phi B^n_\infty)\leq 2\cdot 2^n$
according to Ball \cite{Bal89}.
Therefore,
Proposition~\ref{gammanbP-stab} is implied by the following statement:
If 
\begin{equation}
\label{gammabPeps-stab}
\gamma_n(bP)\geq 
\gamma_n (bB^n_\infty)-\varepsilon
\end{equation}
for $\varepsilon\in(0,\varepsilon_0)$,  where $\varepsilon_0=2^{-68}n^{-28n}$, then there exists $\Phi\in O(n)$ such that
\begin{align}
\label{gammabP-stabvoleps}
\delta_{\rm vol}(P,\Phi B^n_\infty)&\leq 
2^{20}n^{10n}\varepsilon^{\frac14}.
\end{align}
Hence \eqref{gammabP-stabvol} 
follows, since for $\varepsilon\ge \varepsilon_0$ the asserted inequality holds  by the preceding rough bound.

In order to prove \eqref{gammabP-stabvoleps}, the core claim is that  \eqref{gammabPeps-stab} yields
a $\Phi\in O(n)$ such that
\begin{equation}
\label{Pclosetocube}
P\subset (1+\aleph\sqrt[4]{\varepsilon}) \Phi B^n_\infty
\end{equation}
for $\aleph=2^{14}n^{6n}$.

As at the beginning of the proof of Theorem \ref{ell-sym-Loewner-stab},  applying Lemma~\ref{ciuibig} with $v_i=\sqrt{c_i}\,u_i$,
and using that $\binom{k }{ n}\le  (en)^n$,
we may assume (after re-indexing if needed), that
\begin{equation}
\label{prodcu1n}
c_1\cdots c_n\cdot\Big|\det[u_1,\ldots,u_n]\Big|^2\geq (en)^{-n}.
\end{equation}

We set
\begin{equation}
\label{etadefepsilon}
\eta=2^{14}n^{4n}\cdot\varepsilon^{1/4}<\frac18,
\end{equation}
and claim that if $i\in\{1,\ldots,k\}$, then
\begin{equation}
\label{uconsitionBLRBLgsg}
\mbox{$c_i\leq \eta^2$, or there exists some $j\in\{1,\ldots,n\}$ with
$|\langle u_i,u_j\rangle|\geq \cos\eta$.}
\end{equation}
We suppose, on the contrary, that \eqref{uconsitionBLRBLgsg} does not hold and seek a contradiction. Hence we may assume
$$
c_{n+1}>\eta^2\mbox{ \ and \ }|\langle u_j,u_{n+1}\rangle|< \cos\eta
\mbox{ \ for $j=1,\ldots,n$,}
$$
Now  $u_{n+1}=\sum_{i=1}^{n}\lambda_iu_i$, where $\lambda_1,\ldots,\lambda_{n}\in\R$ are uniquely determined. 
We may assume, by reordering, that $|\lambda_1|\geq\ldots\geq|\lambda_n|$. Choose $\theta_i\in\{-1,1\}$ such that $\theta_i\lambda_i=|\lambda_i|$ for $i=1,\ldots,n$, and hence
$u_{n+1}=\sum_{i=1}^{n}|\lambda_i|\theta_iu_i$. As $u_{n+1}\not\in{\rm conv}\{o,\theta_1u_1,\ldots,\theta_nu_n\}$, we have
$\sum_{i=}^n|\lambda_i|>1$, thus
 $|\lambda_1|\geq \frac1{n}$.
Therefore, $c_{n+1}>\eta^2$, $c_{1}\leq 1$ (due to \eqref{ciatmost10}) and \eqref{prodcu1n} imply
\begin{equation}
\label{prod2n+2}
c_2\cdots c_{n+1}\det[u_2,\ldots,u_{n+1}]^2\geq
(en)^{-n}\cdot\frac{\eta^2}{n^2}=
e^{-n}n^{-n-2}\eta^2.
\end{equation}

Next we observe that
$\angle (u_1-u_{n+1},u_1+u_{n+1})=\frac{\pi}2$, and
let $\tilde{w}\in S^{n-1}\cap L$
for $L={\rm lin}\{u_1,u_{n+1}\}$ satisfy that
$\angle (\tilde{w},u_1-u_{n+1})=\angle (\tilde{w},u_1+u_{n+1})=\frac{\pi}4$.
If $v\in S^{n-1}$ with $\angle (v,\tilde{w})\leq\frac{\pi}8$, then the orthogonal projection $v'$ of $v$ to $L$ satisfies that $\|v'\|\geq \cos\frac{\pi}8>\frac12$ and
$\angle (v',\tilde{w})\leq\frac{\pi}8$; therefore, $|\langle u_1,u_{n+1}\rangle|< \cos\eta$ and
Lemma~\ref{absScalarProd} yield that
\begin{align}
\label{diff1n+1}
\Big|\,\left|\left\langle v,u_1\right\rangle\right|
-\left|\left\langle v,u_{n+1}\right\rangle\right|\,\Big|&=\Big|\,\left|\left\langle v',u_1\right\rangle\right|
-\left|\left\langle v',u_{n+1}\right\rangle\right|\,\Big|>\eta/8,
\end{align}
provided $\angle (v,\tilde{w})\leq\frac{\pi}8$ for $v\in S^{n-1}$.

Now there exists a $w\in L\cap S^{n-1}$ with $\angle (w,\tilde{w})\leq\frac{\pi}{16}$ such that $|\frac{\pi}2-\angle (w,u_1)|\geq\frac{\pi}{32}$ and $|\frac{\pi}2-\angle (w,u_{n+1})|\geq\frac{\pi}{32}$, and we consider the ball
$$
\Xi=\frac{w}8+2^{-9}B^n_2=\frac{w}8+\frac1{8\cdot 64}\,B^n_2\subset bP.
$$
If $x\in \Xi$, then  
$\frac1{16}\leq \|x\|\leq \frac14$ and $\angle(x,w)\le \frac{\pi}{64}$, and hence also 
$|\frac{\pi}2-\angle (x,u_1)|\geq\frac{\pi}{64}$ and $|\frac{\pi}2-\angle (x,u_{n+1})|\geq\frac{\pi}{64}$. From the choice of $w$ 
and $\frac1{16}\cdot\frac1{64}=2^{-10}$ we deduce that if $x\in\Xi$, then
\begin{equation}
\label{xXiu1un+1}
\begin{array}{rcll}
|\langle x,u_1\rangle|&\geq& 2^{-10}, &  \\[1ex]
|\langle x,u_{n+1}\rangle|&\geq& 2^{-10}, &\\[1ex]
|\langle x,u_i\rangle|&\leq& 2^{-2},&i=1,\ldots,k.
\end{array}
\end{equation}
We have $\angle (x,\tilde{w})\leq\frac{\pi}{8}$ as $\angle (x,w)\leq\frac{\pi}{64}$, and hence
 \eqref{diff1n+1} yields that
\begin{equation}
\label{xXiu1minusun+1}
\Big|\,\left|\left\langle x,u_1\right\rangle\right|
-\left|\left\langle x,u_{n+1}\right\rangle\right|\,\Big|>2^{-7}\eta
\mbox{ \ \ \ for $x\in\Xi$.}
\end{equation}

As in the proof of Proposition~\ref{gammanbP-prop}, 
let
$$
f_{(b)}(t)=
\left\{
\begin{array}{cl}
\frac1{\sqrt{2\pi}}\,e^{-\frac{t^2}2},&\mbox{ if $|t|\leq b$},\\[1ex]
0,&\mbox{ if $|t|> b$.}
\end{array}
\right.
$$
Hence ({\it cf.} \eqref{gammabPright} and \eqref{gammabPleft})
\begin{align}
\label{gammabPright0}
\gamma_n ([-b,b]^n)&=
\prod_{i=1}^k\left(\int_{\R}f_{(b)}(t)\,\dint t\right)^{c_i},\\
\label{gammabPleft0}
\gamma_n(bP)&=
\int_{\R^n}\prod_{i=1}^kf_{(b)}(\langle x,u_i\rangle)^{c_i}\,\dint x.
\end{align}
We consider the probability density ({\it cf.} Section~\ref{secAuxiliaryJohn})
$$
\tilde{f}_{(b)}=\frac1{\Gamma_b}\,f_{(b)} \mbox{ \ for }
\Gamma_b=\int_{\R}f_{(b)}(s)\, \dint s=\int_{-b}^bf_{(b)}(s)\, \dint s<1,
$$
and the corresponding transport map $\varphi_{b}:(-b,b)\to\R$ such that
\[
\frac1{\sqrt{2\pi}}\int_{-\infty}^{\varphi_{b}(t)} e^{-\frac{s^2}2}\,\dint s
=\int_{-b}^t \tilde{f}_{(b)}(s)\, \dint s.
\]

It follows from Lemma~\ref{propphib}, \eqref{xXiu1un+1} and $0.049>2^{-5}$ that if $x\in\Xi$, then
\begin{equation}
\label{xXiu1un+1phi}
\begin{array}{rcll}
|\varphi_b(\langle x,u_i\rangle)|&\leq& 1,&i=1,\ldots,k,\\
\varphi'_b(\langle x,u_i\rangle)&\geq& 1,&i=1,\ldots,k,\\
\varphi'_b(\langle x,u_i\rangle)&\leq& 2,&i=1,\ldots,k,\\
\varphi_b''(t)&\geq& 2^{-15},&
\mbox{if $t$ is between $|\langle x,u_1\rangle|$ and $|\langle x,u_{n+1}\rangle|$.}
\end{array}
\end{equation}
Therefore, \eqref{xXiu1minusun+1}
yields that if $x\in\Xi$, then
\begin{align}
\label{xXiu1minusun+1phi}
&\Big|\varphi'_b\left(\left\langle x,u_1\right\rangle\right)
-\varphi'_b\left(\left\langle x,u_{n+1}\right\rangle\right)\Big|\nonumber\\
&\quad =
\Big|\varphi'_b\left(\left|\left\langle x,u_1\right\rangle\right|\right)
-\varphi'_b\left(\left|\left\langle x,u_{n+1}\right\rangle\right|\right)\Big|>2^{-22}\eta.
\end{align}
We apply Lemma~\ref{Ball-Barthe-stab}, the stability version 
of the Ball--Barthe inequality \eqref{BallBarthe}, with
$\beta=n^{-4n}\eta^2<e^{-n}n^{-n-2}\eta^2$, 
based on \eqref{prodcu1n} and \eqref{prod2n+2}, and using the estimates 
\eqref{xXiu1un+1phi} and
\eqref{xXiu1minusun+1phi}
to conclude that if $x\in\Xi$, then
\begin{align}
\nonumber
&\prod_{i=1}^k \varphi'_b(\langle x,u_i\rangle)^{c_i}\\
&\nonumber\leq
\det\left( \sum_{i=1}^k\varphi'_b(\langle x,u_i\rangle)c_iu_i\otimes u_i\right) \left( 1+\frac{\beta(\varphi'_b(\langle x,u_1\rangle)-\varphi'_b(\langle x,u_{n+1}\rangle))^2}{4(\varphi'_b(\langle x,u_1\rangle)+\varphi'_b(\langle x,u_{n+1}))^2}\right)^{-1} \\
\label{Ball-Barthe-Xi}
&\leq \det\left( \sum_{i=1}^k\varphi'_b(\langle x,u_i\rangle)c_iu_i\otimes u_i\right)
-2^{-51}n^{-4n}\eta^4,
\end{align}
where we also used that
$(1+s)^{-1}<1-\frac{s}2$ if $s\in(0,\frac12)$ and
the
Ball--Barthe inequality \eqref{BallBarthe} and 
\eqref{xXiu1un+1phi} imply that
$$
\det\left( \sum_{i=1}^k\varphi'_b(\langle x,u_i\rangle)c_iu_i\otimes u_i\right)\geq \prod_{i=1}^k \varphi'_b(\langle x,u_i\rangle)^{c_i}\geq 1.
$$

We observe that the ${\cal C}$ in \eqref{BLstep1} is just $bP$,
and use \eqref{Ball-Barthe-Xi} for  $x\in\Xi$ instead of \eqref{BallBarthe} in \eqref{BLstep1}. 
We deduce from \eqref{gammabPright0}
and \eqref{gammabPleft0} that
\begin{align*}
&\frac{\gamma_n(bP)}{\gamma_n ([-b,b]^n)}=
\int_{\R^n}\prod_{i=1}^k\tilde{f}_{(b)}(\langle u_i,x\rangle)^{c_i}\,\dint x\\
&
=
\int_{bP}\left(\prod_{i=1}^k\gamma_1(\varphi_b(\langle u_i,x\rangle))^{c_i}\right)
\left(\prod_{i=1}^k\varphi_b'(\langle u_i,x\rangle)^{c_i}\right)\,\dint x\\
&\leq  \frac{1}{(2\pi)^{\frac{n}2}}\int_{bP\setminus\Xi}\left(\prod_{i=1}^ke^{-c_i\varphi_b(\langle u_i,x\rangle)^2/2}\right)
\det\left(\sum_{i=1}^kc_i\varphi_b'(\langle u_i,x\rangle )\,u_i\otimes u_i\right)\,\dint x \,+\\
&\quad +\frac{1}{(2\pi)^{\frac{n}2}}\int_{\Xi}\left(\prod_{i=1}^ke^{-c_i\varphi_b(\langle u_i,x\rangle)^2/2}\right)\times\\
&\quad\qquad\times\left(
\det\left(\sum_{i=1}^kc_i\varphi_b'(\langle u_i,x\rangle )\,u_i\otimes u_i\right)
-2^{-51}n^{-4n}\eta^4\right)\,\dint x.
\end{align*}
Here $V(\Xi)=2^{-9n}\frac{\pi^{n/2}}{\Gamma(\frac{n}2+1)}>2^{-9n}n^{-\frac{n}2-\frac12}\frac{(e\pi)^{n/2}}{4}$, and \eqref{xXiu1un+1phi} yields that if $x\in\Xi$, then
$$
\prod_{i=1}^ke^{-c_i\varphi(\langle u_i,x\rangle)^2/2}\geq
e^{-\frac12\sum_{i=1}^kc_i}=e^{-n/2};
$$
therefore, we deduce from \eqref{BLstep3} that
$$
\frac{\gamma_n(bP)}{\gamma_n ([-b,b]^n)}\leq 1-\frac{V(\Xi)}{(2\pi)^{\frac{n}2}}\cdot e^{-n/2}2^{-51}n^{-4n}\eta^4
<1-2^{-53}n^{-15n}\eta^4.
$$
We have  $\gamma_n ([-b,b]^n)=(2\Phi(b)-1)^n\geq (2\Phi(1)-1)^n>2^{-n}$ for the
cumulative distribution function
$\Phi(t)=\frac1{\sqrt{2\pi}}\int_{-\infty}^{t} e^{-\frac{s^2}2}\,\dint s$, and hence, by \eqref{gammabPeps-stab}
$$
\gamma_n ([-b,b]^n)-
\varepsilon<\gamma_n(bP)<\gamma_n ([-b,b]^n)-
2^{-53}n^{-16n}\eta^4.
$$
These inequalities contradict \eqref{etadefepsilon}, and in turn prove the claim
\eqref{uconsitionBLRBLgsg}.

\medskip 

Since $\eta<1/(3\sqrt{k})$ and due to \eqref{uconsitionBLRBLgsg},  Corollary~\ref{almostort} 
yields that there exists an orthonormal basis $w_1,\ldots,w_n$
such that 
\begin{equation}
\label{uiwidistance}
\angle (u_i,w_i)<3\sqrt{k}\,\eta<3n\eta \mbox{ \ for }i=1,\ldots,n. 
\end{equation}
Let $\Phi\in O(n)$ satisfy that $\Phi B^n_\infty$ is the cube $\{x\in\R^n:|\langle x,w_i\rangle|\leq 1\}$.
For $x\in P$, 
 \eqref{PJonhHomothety} and \eqref{uiwidistance} imply that
if $i=1,\ldots n$, then
$$
|\langle x,w_i\rangle|\leq
|\langle x,u_i\rangle|+
|\langle x,w_i-u_i\rangle|\leq
1+3n^{3/2}\eta,
$$
and hence
$$
P\subset (1+\aleph\sqrt[4]{\varepsilon}) \Phi B^n_\infty
$$
for $\aleph=2^{14}n^{6n}$,
as claimed in \eqref{Pclosetocube}. 

Since $(1+t)^n<e^{nt}<1+2nt$ for $0<t<1/n$, it follows that $(1+\aleph\sqrt[4]{\varepsilon})^n<1+2n\aleph\sqrt[4]{\varepsilon}$ if $\varepsilon< \varepsilon_0$. As $V(b\Phi B^n_\infty)\leq 4^n$, we observe that 
\begin{align}
\nonumber
\gamma_n\left((1+\aleph\sqrt[4]{\varepsilon}) b\Phi B^n_\infty\setminus ( b\Phi B^n_\infty)\right)&\leq
(2\pi)^{-n/2}V\left((1+\aleph\sqrt[4]{\varepsilon}) b\Phi B^n_\infty\setminus ( b\Phi B^n_\infty)\right)\\
\label{gammanbigcubediff}
&\leq
(2\pi)^{-n/2}2n\aleph\sqrt[4]{\varepsilon}4^n.
\end{align}

We have $V(P)\leq V(B^n_\infty)\leq V(\Phi B^n_\infty)$ according to Ball \cite{Bal89}, and hence $2V(\Phi B^n_\infty\setminus P)\geq  \delta_{\rm vol}(\Phi B^n_\infty,P)$. In addition, 
$e^{-\|x\|^2/2}\geq e^{-2n}$ if $x\in bP\cup b\Phi B^n_\infty$ by \eqref{PJonhHomothety}; therefore,
\eqref{Pclosetocube} and 
\eqref{gammanbigcubediff}, and finally \eqref{gammabPeps-stab} 
and $\aleph=2^{14}n^{6n}$ yield
\begin{align*}
 \delta_{\rm vol}(\Phi B^n_\infty,P)&
 \leq 2 V(b\Phi B^n_\infty\setminus (bP))
 \leq
 2\cdot (2\pi)^{n/2}e^{2n}\gamma_n(b\Phi B^n_\infty\setminus (bP))\\
&\leq
 2\cdot (2\pi)^{n/2}e^{2n}\gamma_n\left((1+\aleph\sqrt[4]{\varepsilon})b\Phi B^n_\infty\setminus (bP)\right)\\
 &\leq
 2\cdot (2\pi)^{n/2}e^{2n}\left(\gamma_n\left((1+\aleph\sqrt[4]{\varepsilon})b\Phi B^n_\infty\right)
 -\gamma_n\left(b\Phi  B^n_\infty\right)+\varepsilon\right)\\
 &=
 2\cdot (2\pi)^{n/2}e^{2n}\left(\gamma_n\left((1+\aleph\sqrt[4]{\varepsilon})b\Phi B^n_\infty\setminus (b\Phi B^n_\infty)\right)+\varepsilon\right)\\
 &\le 2\cdot (2n)\cdot e^{2n}\cdot \aleph\cdot  4^n
  \cdot \sqrt[4]{\varepsilon}+2\cdot 2^{2n}\cdot 2^{3n}\varepsilon\le 2^{20}\cdot n^{10n}\sqrt[4]{\varepsilon},
 \end{align*}
 which proves \eqref{gammabP-stabvoleps}.
\hfill\proofbox
\end{proof}

\medskip

The following lemma allows us to compare the Hausdorff distance and the volume difference. The lemma is a straightforward consequence of the proof of a  result due to Groemer \cite[Theorem (ii)]{Groemer2000} in the case of symmetric convex bodies.

\begin{lemma}
\label{deltaHvolCube}
Let $K,L\subset\R^n$ be $o$-symmetric convex bodies. If $rB^n_2\subset K,L\subset R B^n_2$, then
$$
\delta_{\rm H}(K,L)\le \left(\frac{n}{\kappa_{n-1}}\right)^{\frac{1}{n}}\left(\frac{R}{r}\right)^{\frac{n-1}{n}}\delta_{\rm vol}(K,L)^{\frac{1}{n}}.
$$
If $r=1,R=\sqrt{n}$, then 
$$
\delta_{\rm H}(K,L)\le n\cdot n^{\frac{1}{2n}}
\delta_{\rm vol}(K,L)^{\frac{1}{n}}.
$$
\end{lemma}
\begin{proof}
The first assertion follows from the proof 
of \cite[Theorem (ii)]{Groemer2000}, where in the case of symmetric convex bodies instead of the diameter of $K,L$ the circumradius can be used in the argument. 

The second part follows from the first part by basic calculus, using  that 
$$
\Gamma\left(\frac{n+1}{2}\right)\le \sqrt{\pi(n-1)}\left(\frac{n-1}{2e}\right)^{\frac{n-1}{2}}\frac{3}{2}
$$
for $n\ge 2$.
\hfill\proofbox
\end{proof}

\medskip 

Now we are ready to prove a stability version of Theorem~\ref{ell-sym-John}.

\begin{theo}
\label{ell-sym-John-stab}
If $K$ is an origin symmetric convex body in $\R^n$ such that $B^n_2\subset K$ is the John ellipsoid of $K$, then
there exists $\Phi\in O(n)$   such that
\begin{align}
\label{ellKBinfJohnvol}
\ell(K)&\geq 
\ell (B^n_\infty)+2^{-88}n^{-40n}\delta_{\rm vol}(K,\Phi B^n_\infty)^4,\\
\label{ellKBinfJohnH}
\ell(K)&\geq 
\ell(B^n_\infty)+2^{-90}n^{-44n}\delta_{\rm H}(K,\Phi  B^n_\infty)^{4n}.
\end{align}
\end{theo}

\begin{proof}
According to John's  characteristic condition \eqref{John-iso},
there exist $u_1,\ldots,u_k\in S^{n-1}\cap \partial K$ and $c_1,\ldots,c_k>0$ with $n\leq k \leq \frac{n(n+1)}2$ that
satisfy \eqref{ciui}. 
 
For
$P=\{x\in\R^n:\,|\langle x,u_i\rangle|\leq 1,\;i=1,\ldots,k\}$,
we have 
\begin{equation}
\label{KinPJohn}
\mbox{$K\subset P\subset\sqrt{n}B^n_2$ and $\ell(tK)\geq \ell(tP)$  for $t>0$.}
\end{equation}
We deduce from \eqref{ellGaussianint} that
 \begin{align}
\label{ellGaussianintPK}
\ell(K)-\ell(P)&=\int_0^\infty(\gamma_n(tP)-\gamma_n(tK))\,\dint t.
\end{align}
To estimate $\ell(P)$, first observe that \eqref{gammabP} yields 
\begin{equation}
\label{gammabP0}
\gamma_n(tP)\leq \gamma_n 
(tB^n_\infty)\mbox{ \ for $t>0$.}
\end{equation}
Next, let $\Phi\in O(n)$   be such that
\begin{align*}
\delta_{\rm vol}(P,\Phi  B^n_\infty)&=\min_{\Phi'\in O(n)}\delta_{\rm vol}(P,\Phi' B^n_\infty).
\end{align*}
Hence if $t\in[1,2]$, then
 Proposition~\ref{gammanbP-stab} implies that
\begin{align}
\label{gammabP-stabvol0}
\gamma_n(tP)&\leq 
\gamma_n (tB^n_\infty)-2^{-80}n^{-40n}\delta_{\rm vol}(P,\Phi B^n_\infty)^4.
\end{align}
Combining \eqref{ellGaussianintPK}, 
\eqref{gammabP0} and \eqref{gammabP-stabvol0}
, we obtain
\[ 
\ell(P)\geq 
\ell (B^n_\infty)+2^{-80}n^{-40n}\delta_{\rm vol}(P,\Phi B^n_\infty)^4.
\]
In the following, we use again  \cite{Bal89} to get
\begin{equation}
\label{PBinftyBall}
V(K)\leq V(P)\leq V(B^n_\infty)=V(\Phi B^n_\infty).
\end{equation}

We finish the proof by distinguishing two cases:
\begin{itemize}
\item 
If $V(\Phi B^n_\infty\setminus K)\leq 2V(\Phi B^n_\infty\setminus P)$, then \eqref{PBinftyBall} implies that
$$
\delta_{\rm vol}(K,\Phi B^n_\infty)\leq 2 V(\Phi B^n_\infty\setminus K)\leq 4 V(\Phi B^n_\infty\setminus P)\leq 4\delta_{\rm vol}(P,\Phi B^n_\infty),
$$
and hence
\begin{equation}
\label{ellKifKPclose}
\ell(K)\geq 
\ell (B^n_\infty)+2^{-88}n^{-40n}\delta_{\rm vol}(K,\Phi B^n_\infty)^4.
\end{equation}

\item 
If $V(\Phi B^n_\infty\setminus K)\geq 2V(\Phi B^n_\infty\setminus P)$ and $t\in[1,2]$, then,
as $e^{-\frac12\|x\|^2}\geq e^{-2n}$ by
\eqref{KinPJohn} for  $x\in tP$, we obtain
\begin{align*}
\gamma_n(tP)-\gamma_n(tK)&\geq 
\gamma_n\Big((t\Phi B^n_\infty)\cap[(tP)\setminus(tK)]\Big)\\&\geq (2\pi)^{-\frac{n}2}e^{-2n}\cdot
V\Big((t\Phi B^n_\infty)\cap[(tP)\setminus(tK)]\Big)\\
&\geq \frac{(2\pi)^{-\frac{n}2}e^{-2n}}2\cdot V(\Phi B^n_\infty\setminus K)\\
&\geq 
\frac{(2\pi)^{-\frac{n}2}e^{-2n}}4\cdot
\delta_{\rm vol}(K,\Phi B^n_\infty).
\end{align*}
For the third inequality, we used that
\begin{align*}
 V\left(\Phi B^n_\infty\setminus K\right)&= V\left(\Phi B^n_\infty\cap P\setminus K\right)  + V\left(\Phi B^n_\infty\cap P^{\sf c}\setminus K\right)\\
 &\le V\left(\Phi B^n_\infty\cap P\setminus K\right)  + V\left(\Phi B^n_\infty\setminus P \right)\\
 &\le V\left(\Phi B^n_\infty\cap P\setminus K\right)  + \frac{1}{2} V\left(\Phi B^n_\infty\setminus K \right),
\end{align*}
the last inequality follows since $V(K)\le V(\Phi B^n_\infty)$. 

Hence 
$$
\gamma_n(tP)-\gamma_n(tK)\geq 2^{-6n}\delta_{\rm vol}(K,\Phi B^n_\infty)\geq n^{-6n}\delta_{\rm vol}(K,\Phi B^n_\infty).
$$
From \eqref{ellGaussianintPK} we deduce that
\begin{align}
\label{ellKifKPfar}
\ell(K)-\ell (B^n_\infty)
&\geq \ell(K)-\ell(P)\geq 
 n^{-6n}\delta_{\rm vol}(K,\Phi B^n_\infty)\nonumber\\
 &\ge 
 2^{-88}n^{-40n}\delta_{\rm vol}(K,\Phi_v B^n_\infty)^4.
\end{align}
\end{itemize}
We conclude \eqref{ellKBinfJohnvol} from \eqref{ellKifKPclose} and \eqref{ellKifKPfar}.

Finally, \eqref{ellKBinfJohnH} is implied by \eqref{ellKBinfJohnvol} and Lemma \ref{deltaHvolCube}.
\hfill\proofbox
\end{proof}

\section{Proof of Theorem~\ref{Lowner-stab} and Theorem~\ref{John-stab} }
\label{secTheorem1314}

The following trivial observation relates Hausdorff distance to the ``dilation distance".

\begin{lemma}
\label{deltaHdilate}
For convex bodies $K,C\subset\R^n$, if $n^{\frac{-1}2}B^n_2\subset K,C\subset \sqrt{n}\,B^n_2$, then
$$
(1+\sqrt{n}\delta_{\rm H}(K,C))^{-1}C\subset K\subset (1+\sqrt{n}\delta_{\rm H}(K,C))C,
$$
and if $(1+t)^{-1}C\subset K\subset (1+t)C$ for $t\geq 0$, then $\delta_{\rm H}(K,C)\leq \sqrt{n}\,t$.
\end{lemma}
\begin{proof} 
We use that $\delta_{\rm H}(K,C)$ is the minimum of $\varrho\geq 0$ such that $C\subset K+\varrho\,B^n_2$ and $K\subset C+\varrho\,B^n_2$. \hfill\proofbox
\end{proof}

\medskip 

\noindent{\it Proof of Theorem~\ref{Lowner-stab} and Theorem~\ref{John-stab}.} The statements about the $\ell$-norm follow from Theorem~\ref{ell-sym-Loewner-stab} and Theorem~\ref{ell-sym-John-stab}. 
In turn,  the statements about the mean width follow from polarity, \eqref{mean-ell} 
and Lemma~\ref{deltaHdilate}.
\hfill\proofbox
\

\section{Proof of Theorem~\ref{meanw-iso-measure-stab}}
\label{secTheorem15}

We need some auxiliary statements. The first lemma is a counterpart to \cite[Lem.~10.1]{BoH17} for even isotropic  measures.

\begin{lemma}
\label{JohnLimited}
If $\mu$ is an even isotropic measure on $S^{n-1}$, $n\geq 2$, then there exists a discrete even  isotropic measure $\mu_0$ on $S^{n-1}$ such that ${\rm supp}\,\mu_0\subset {\rm supp}\,\mu$ and  $|{\rm supp}\,\mu_0|\leq n(n+1)/2$.
\end{lemma}

\begin{proof} We set $d=\frac{n(n+1)}2$, and consider the $d$-dimensional real vectorspace $\mathcal{M}^d$ of symmetric  $n\times n$ matrices. Basic elements of $\mathcal{M}^d$ are the $n\times n$ identity matrix $\Id_n$, and the rank one matrices $u\otimes u=uu^\top$ for a $u\in S^{n-1}$.
We equip $\mathcal{M}^d$ with a scalar product; namely, 
$\langle A,B\rangle ={\rm tr}\,AB^\top $ for $A,B\in \mathcal{M}^d$, and hence if
 $A=[a_{ij}]$ and $B=[b_{ij}]$, then $\langle A,B\rangle=\sum_{i,j=1,\ldots,n}a_{ij}b_{ij}$. We claim that
\begin{equation}
\label{JohnLimited-conv}
\frac1n\,\Id_n\in {\rm conv} \{u\otimes u:\,u\in {\rm supp}\,\mu\}.
\end{equation}
It is equivalent to prove that for any $M\in\mathcal{M}^d$, we have
\begin{equation}
\label{JohnLimited-conv0}
\frac1n\,\langle \Id_n,M\rangle \leq \max \{\langle u\otimes u,M\rangle:\,u\in {\rm supp}\,\mu\}.
\end{equation}
What is known about $\mu$ is that $\frac1n\,\mu$ is a probability measure, and  
$$
\frac1n\int_{S^{n-1}}\langle u\otimes u,M\rangle\,d\mu(u)=\frac1n\,\langle I_n,M\rangle,
$$
which in turn yields \eqref{JohnLimited-conv0}, and hence also \eqref{JohnLimited-conv}.

Writing $\mathcal{M}^{d-1}_1$ to denote the affine $(d-1)$-dimensional subspace of $\mathcal{M}^d$ consisting of matrices with trace $1$, we observe that  $u\otimes u\in \mathcal{M}^{d-1}_1$ for $u\in {\rm supp}\,\mu$ and $\frac1n\,\Id_n\in \mathcal{M}^{d-1}_1$. 
According to the Charatheodory theorem applied to \eqref{JohnLimited-conv} in $\mathcal{M}^{d-1}_1$, there exist $u_1,\ldots,u_k\in {\rm supp}\,\mu$
with $u_j\neq \pm u_i$ for $i\neq j$ and $k\leq d$ such that
$$
\frac1n\,\Id_n\in {\rm conv} \{u_i\otimes u_i:\,i=1,\ldots,k\}.
$$
It follows that there exist $\tilde{c}_1,\ldots\tilde{c}_k\geq 0$ with $\tilde{c}_1+\cdots+\tilde{c}_k=1$ such that
$$
\sum_{i=1}^kn\tilde{c}_i \,u_i\otimes u_i=\Id_n.
$$
Therefore, we can define the even measure $\mu_0$ so that ${\rm supp}\,\mu_0\subset \{\pm u_1,\ldots,\pm u_k\}$, and $\mu_0(u_i)=n\tilde{c}_i/2$.
\hfill\proofbox
\end{proof}

\medskip 

The following lemma implies the bounds involving the $\delta_{\rm WO}$ distance in Theorem~\ref{meanw-iso-measure-stab}, once the corresponding bounds for the $\delta_{\rm HO}$ are established.

\begin{lemma}
\label{HW}
For any isotropic measures $\mu$ and $\nu$ on $S^{n-1}$, we have
\begin{equation}
\label{HW-eq}
\delta_{\rm W}(\mu,\nu)\leq 7\pi n^3 \delta_{\rm H}({\rm supp}\,\mu,{\rm supp}\,\nu).
\end{equation}
\end{lemma}
\begin{proof} If $\delta_{\rm H}({\rm supp}\,\mu,{\rm supp}\,\nu)\leq \frac1{7n^2}$, then   \cite[Cor.~6.2]{BFH19} yields that
$$
\delta_{\rm W}(\mu,\nu)\leq 2n \cdot \delta_{\rm H}({\rm supp}\,\mu,{\rm supp}\,\nu).
$$
On the other hand, for any $f\in{\rm Lip}_1(S^{n-1})$, we may assume that $-\frac{\pi}2\leq f(u)\leq \frac{\pi}2$ for $u\in S^{n-1}$. 
As $\mu(S^{n-1})=\nu(S^{n-1})=n$, we have  $\delta_{\rm W}(\mu,\nu)\leq \pi n$. Therefore, if $\delta_{\rm H}({\rm supp}\,\mu,{\rm supp}\,\nu)\geq \frac1{7n^2}$, then \eqref{HW-eq} readily holds.
\hfill\proofbox
\end{proof}

\medskip 

In the next geometric lemma, $e_1,\ldots,e_n$ denotes the standard basis of $\R^n$ (or any orthonormal basis). 

\begin{lemma}\label{Le7.3geom} 
Let $\varepsilon\in (0,1/2)$ and $x\in \conv\{e_1,\ldots,e_n\}$. If $\angle(x,e_i)\ge \varepsilon$ for $i=1,\ldots,n$, then $\|x\|\le 1-4^{1-n}\cdot\varepsilon$.    
\end{lemma}

\begin{proof} We proceed by induction on the dimension $n\ge 2$. Let $n=2$. Then there are $t\in (0,1)$, $s\in [\varepsilon,\pi/2-\varepsilon]$ and $\lambda\in (0,1)$ such that
$$
x=\lambda\left(\cos(s)e_2+\sin(s)e_1\right)=(1-t)e_2+t e_1,
$$
hence
$$
\lambda=\left(\sin(s)+\cos(s)\right)^{-1}\le \left(\sin(\varepsilon)+\cos(\varepsilon)\right)^{-1}\le 1-4^{-1}\varepsilon,
$$
which implies that $\|x\|=\lambda\le 1-4^{1-2}\cdot\varepsilon$.

Now we assume that $n\ge 3$ and that the assertion holds in smaller dimensions. Let $x$ be as in the statement of the lemma. Then there are $t\in (0,1)$ and $e\in \conv\{e_1,\ldots,e_{n-1}\}$ such that $x=(1-t)e_n+te$. We distinguish two cases. 

Case 1: $\angle (e,e_i)\ge \varepsilon/2$ for $i=1,\ldots,n-1$. An application of the induction hypothesis to $e$ in the linear subspace spanned by $e_1,\ldots,e_{n-1}$ then yields that $\|e\|\le 1-4^{2-n}\cdot\varepsilon/2$. If $t\in [1/2,1)$, then
$$
\|x\|\le 1-t+t\|e\|\le 1-t+t-\frac{1}{2}4^{2-n}t\varepsilon \le 1-4^{1-n}\varepsilon .
$$
Now let $t\in (0,1/2)$. Since $\angle(x,e_n)\ge \varepsilon$, we have
$$
\left(1-\frac{1}{4}\varepsilon^2\right)^{\frac{1}{2}}\ge 1-\frac{1}{4}\varepsilon^2\ge\cos(\varepsilon)\ge \langle x,e_n\rangle\ge \frac{1-t}{\sqrt{(1-t)^2+t^2}},
$$
hence
$$
t^2\ge \frac{1}{4}\varepsilon^2\left(1-2t+2t^2\right)\ge \frac{1}{8}\varepsilon^2.
$$
This shows that $t\ge \frac{1}{4}\varepsilon$. Therefore
\begin{align*}
\|x\|^2&\le (1-t)^2+t^2=1-t\cdot 2(1-t)\le 1-t\le 1-\frac{1}{4}\varepsilon,
\end{align*}
which leads to
$$
\|x\|\le 1-\frac{1}{16}\varepsilon \le 1-4^{1-n}\varepsilon.
$$

Case 2: $\angle (e,e_i)< \varepsilon/2$ for some $i\in\{1,\ldots,n-1\}$. We may assume that $\angle (e,e_1)< \varepsilon/2$. Then $\angle (x,e_1)\ge \varepsilon$ implies that $\angle (x,e)\ge \varepsilon/2$. In the two-dimensional subspace spanned by $e_n$ and $e$, we define $\tilde{x}$ by $\{\tilde{x}\}=[0,\infty)x\cap \conv\{e_n,\|e\|^{-1}e\}$. Since $\angle (x,e_n)=\angle (\tilde{x},e_n)\ge \varepsilon/2$ and $\angle (x,e)=\angle (\tilde{x},e)\ge \varepsilon/2$, we conclude that
$$
\|x\|\le \|\tilde x\|\le 1-\frac{1}{4}\varepsilon/2\le 1-4^{1-n}\varepsilon,
$$
which proves the assertion.
\hfill\proofbox
\end{proof}

\medskip 

The statements (a) and (d) of Theorem ~\ref{meanw-iso-measure-stab} are  implied by the following theorem. The estimates in terms of the Wasserstein distance in Theorem~\ref{meanw-iso-measure-stab} follow from Lemma~\ref{HW}. By $\nu$ we denote a cross measure on $S^{n-1}$.

\begin{theo}\label{theo:theorem 7.4}
Let $\mu$ be an even isotropic measure on $S^{n-1}$. Let $c\ge 3$ be an absolute constant as in Theorem~\ref{ell-sym-Loewner-stab}. If 
\begin{equation}\label{eq:7.neu2}
    \ell(Z_\infty(\mu))\ge (1-\varepsilon)\ell(Z_\infty(\nu)),
\end{equation}
for some $\varepsilon\in (0,\varepsilon_0)$ for $\varepsilon_0=\frac{1}{2}n^{-cn^2}$, then 
$$\delta_{\rm HO}({\rm supp}\,\mu,{\rm supp}\,\nu)\le n^{cn^2}\varepsilon.$$    
\end{theo}

\begin{proof} By Lemma~\ref{JohnLimited} there exists a discrete even isotropic measure $\mu_0$ on $S^{n-1}$ such that ${\rm supp}\,\mu_0\subset {\rm supp}\,\mu$ and  $|{\rm supp}\,\mu_0|\leq n(n+1)/2$. Then $Z_\infty(\mu_0)\subset Z_\infty(\mu)$ and $B^n_2$ is the Löwner ellipsoid for $Z_\infty(\mu_0)$ and $Z_\infty(\mu)$ ({\it cf.}~\eqref{John-iso}). Hence, the assumption implies that
$$
 \ell(Z_\infty(\mu_0))\ge (1-\varepsilon)\ell(Z_\infty(\nu)).
$$
The proof of Theorem~\ref{ell-sym-Loewner-stab} shows that there exists an orthogonal transformation $\Phi\in O(n)$ such that 
$$
\delta_{\rm H}({\rm supp}\, \mu_0,\Phi B^n_1)\le n^{cn^2}\varepsilon;
$$
see \eqref{eq:sec2close} and the condition $c\ge c(2)\ge 4c(1)$ at the end of the proof of Theorem~\ref{ell-sym-Loewner-stab}. If $\nu$ is the cross measure corresponding to $\Phi B^n_1=\conv\{\pm w_1,\ldots,\pm w_n\}$ with an orthonormal basis $w_1,\ldots,w_n$ of $\R^n$, then 
$$
{\rm supp}\, \nu\subset {\rm supp}\, \mu +n^{cn^2}\varepsilon B^n_2,
$$
since ${\rm supp}\, \mu_0\subset {\rm supp}\, \mu$. Moreover, we have
$$
{\rm supp}\, \mu_0\subset {\rm supp}\, \nu +n^{cn^2}\varepsilon B^n_2.
$$
In order to show that in fact ${\rm supp}\, \mu\subset {\rm supp}\, \nu +n^{c n^2}\varepsilon B^n_2$, we assume that there is some $z\in {\rm supp}\, \mu\subset Z_\infty(\mu)\subset B^n_2$ and $z\notin{\rm supp}\, \nu +n^{cn^2}\varepsilon B^n_2$,  aiming at a contradiction. If the assumption holds, then $z\in S^{n-1}$ and $\angle(z,w_i)\ge n^{c n^2}\varepsilon\in (0,1/2)$ for $i=1,\ldots,n$. An application of Lemma~\ref{Le7.3geom} shows that $\|z\|_{B^n_1}=1+t$ satisfies
$$
1+t\ge \left(1-4^{1-n}n^{cn^2}\varepsilon\right)^{-1}\ge 1+4^{1-n}n^{cn^2}\varepsilon,
$$
hence $t\ge 4^{1-n}n^{cn^2}\varepsilon$. Similarly as in the final part of the proof of Theorem~\ref{ell-sym-Loewner-stab} it follows that there is a facet $F$ of $B^n_1$ such that $z_0=(1+t)^{-1}z\in F$. Setting $P_z:=\conv\left(\{\pm z\}\cup B^n_1\right)$, we obtain
\begin{equation}\label{eq:7bound3}
V(P_z)-V(B^n_1)\ge 2tV(\conv\{o,F\})\ge \frac{2}{n\cdot n!}t\ell(B^n_1). 
\end{equation}
Using Lemma~\ref{EllDifference} and \eqref{eq:2bound}, we deduce from \eqref{eq:7bound3} that
\begin{equation}\label{eq:7bound4}
\ell(B^n_1)-\ell(P_z)\ge \left(\frac{n}{2\pi e}\right)^{\frac{n}{2}}\left(V(P_z)-V(B^n_1)\right)\ge n^{-4n} t\ell(B^n_1).
\end{equation}
On the other hand, we have
$$
\ell(B^n_1)-\ell(P_z)\le \ell(B^n_1)-(1-\varepsilon)\ell(B^n_1)=\varepsilon \ell(B^n_1),
$$
which together with \eqref{eq:7bound4} implies that $t\le n^{4n}\varepsilon$. But this is in conflict with $t\ge 4^{1-n}n^{cn^2}\varepsilon$, since $c\ge 3$.
    \hfill\proofbox
\end{proof}

\medskip

In the proof of Theorem~\ref{meanw-iso-measure-stab} (b), (c), we need two lemmas. 
The first lemma is a dual counterpart to Lemma \ref{closeToRegCross}.

\begin{lemma}
\label{closesimplex}
Let $u_1,\ldots,u_k\in S^{n-1}$, and let $e_1,\ldots,e_n$ be an orthonormal basis of $\R^n$. Let  $P=\{x\in\R^n:\langle x,u_i\rangle\le 1,\; i=1,\ldots,n\}$ be a polytope, and let $B^n_\infty=\{x\in\R^n:|\langle x,e_i\rangle|\le 1,\; i=1,\ldots,n\}$.  Fix $\eta\in (0,1/(2\sqrt{n}))$. If $\delta_{\rm H}(\{u_1,\ldots,u_k\},\{\pm e_1,\ldots,\pm e_{n}\})\le \eta$,
then
$$
(1-\sqrt{n} \eta) B^n_\infty\subset  P\subset (1+2\sqrt{n} \eta) B^n_\infty.
$$
\end{lemma}

\proof  We obtain 
$$
P^\circ=\conv\{u_1,\ldots,u_k\},\quad 
B^n_1=\conv\{\pm e_1,\ldots,\pm e_{n}\},
$$
and by assumption
$$
P^\circ\subset B^n_1+\eta B^n_2,\quad B^n_1\subset P^\circ+\eta B^n_2.
$$
Since $B^n_\infty\subset\sqrt{n}B^n_2$, we have $(1/\sqrt{n})B^n_2\subset B^n_1$. Hence
$$
\frac{1}{\sqrt{n}}B^n_2\subset P^\circ+\eta B^n_2, \quad \left(\frac{1}{\sqrt{n}}-\eta\right)B^n_2\subset P^\circ,
$$
and hence
$$
B^n_1\subset P^\circ+\frac{\eta}{\frac{1}{\sqrt{n}}-\eta}P^\circ.
$$
Since $\eta<1/(2\sqrt{n})$, we deduce that
$$
B^n_1\subset (1+2\sqrt{n}\eta)P^\circ.
$$
Furthermore, from $B^n_2\subset \sqrt{n}B^n_1$ we deduce that
$$
P^\circ\subset B^n_1+\eta B^n_2\subset B^n_1+\eta \sqrt{n} B^n_1=(1+\sqrt{n}\eta)B^n_1,
$$
and therefore
$$
(1-\sqrt{n}\eta)B^n_\infty\subset (1+\sqrt{n}\eta)^{-1}B^n_\infty \subset P,
$$
which completes the proof.
\hfill \proofbox

\medskip 

The second lemma implies a lower bound for the volume that is cut off from a cube by an additional hyperplane  provided that the hyperplane is not too close to a facet hyperplane of the cube. For $u\in S^{n-1}$, we denote by $H^-(u)=\{x\in\R^n:\langle x,u\rangle\le 1\}$ the halfspace touching $B^n_2$ at $u$ that contains the origin $o$, and $H^+(u)=\{x\in\R^n:\langle x,u\rangle\ge 1\}$ is the closure of the complement of $H^-(u)$. 

\begin{lemma}
\label{volumebound}
Let $u_1,\ldots,u_k\in S^{n-1}$,  let $w_1,\ldots,w_n$ be an orthonormal basis of $\R^n$, and let $\eta>0$. 
Assume that $\angle(u_i,w_i)\le \eta$ for $i=1,\ldots,n$ and $\angle(u_k,\epsilon_iw_i)\ge c(n)\eta$ for  $\epsilon_i\in\{-1,1\}$ and $i=1,\ldots,n$ with $c(n)= 2^{n+4}\cdot n^{n+3}$. 
If 
\begin{align*}
P&=\{x\in \R^n:|\langle x,u_i\rangle|\le 1,\; i=1,\ldots,n\}\cap H^-(u_k)
\quad \text{and}\\
B^n_\infty &=\{x\in\R^n:|\langle x,w_i\rangle|\le 1 ,\; i=1,\ldots,n\},
\end{align*}
then 
$$
V(P)\le\left(1-\frac{\Delta}{2^4\cdot n^{n+1}}\right)V(B^n_\infty),
$$
where $\Delta:=\min\{\angle(u_k,\epsilon_iw_i): \epsilon_i\in\{-1,1\},\;i=1,\ldots,n\}$.
\end{lemma}

\begin{proof} By assumption, $\eta<c(n)^{-1}< 1/(2\sqrt{n})$, since $\angle(u_k,\epsilon_iw_i)<\pi/4$ for some $\epsilon_i\in\{-1,1\}$ and some $i\in\{1,\ldots,n\}$.  
Let $H^+:=H^+(u_k)$. We may assume that $u_k$ is in the positive hull of $w_1,\ldots,w_n$ and $\omega_0:=\angle(u_k,w_1)=\min\{\angle(u_k,w_i):i=1,\ldots,n\}$. Note that $c(n)\eta\le \omega_0\le \omega_0(n)$, where $\cos(\omega_0(n))=n^{-\frac{1}{2}}$. To see this, assume that $\omega_0>\omega_0(n)$, which implies that
$$
1=\|u_k\|^2=\sum_{i=1}^n\langle u_k,w_i\rangle^2<n\cdot\frac{1}{n}=1,
$$
a contradiction.

Let $F_1=B^n_\infty\cap H^+(w_1)$. If $z$ denotes the  point of $ H^+\cap F_1$ closest to $w_1$, then $\|z-w_1\|=\tan\left(\frac{\omega_0}{2}\right)$, and therefore $H^+\cap F_1$ contains an $(n-1)$-dimensional ball of radius $1-\tan\left(\frac{\omega_0}{2}\right)$. Moreover, $H^+\cap B^n_\infty$ contains a point whose distance from the affine hull of $F_1$ is at least
\begin{align}\label{eq:auxin}
f(\omega_0):&=\left(1-\tan\left(\frac{\omega_0}{2}\right)\right)\tan(\omega_0)\nonumber\\
&=\frac{\sin(\omega_0)+\cos(\omega_0)-1}{\cos(\omega_0)}\ge \frac{1}{2}\omega_0,\quad \omega_0\in[0, {\pi}/{2}].
\end{align}
For the proof of \eqref{eq:auxin}, we consider
$$
g(x)=\sin(x)+\cos(x)-1-\frac{1}{2}x\cos(x),\quad x\in[0,\pi/2].
$$
Since $g''(x)=\cos(x)\left(\frac{1}{2}x-1\right)\le 0$ for $ x\in[0,\pi/2]$,  $g$ is a concave function. The assertion follows, since $g(0)=g(\pi/2)=0$. 

From
$$
\tan\left(\frac{\omega_0}{2}\right)\le 
\tan\left(\frac{\omega_0(n)}{2}\right)
=\frac{\sqrt{1-n^{-1/2}}}{\sqrt{1+n^{-1/2}}}<1-\frac{1}{2\sqrt{n}},
$$
we obtain
$$
1-\tan\left(\frac{\omega_0}{2}\right)\ge \frac{1}{2\sqrt{n}}.
$$
We thus conclude that
\begin{align*}
V(H^+\cap B^n_\infty)&\ge \frac{1}{n}\kappa_{n-1}\left(\frac{1}{2\sqrt{n}}\right)^{n-1}\frac{1}
{2}\omega_0    
\ge 2^{-3} n^{-\frac{1}{2}}n^{-n}\omega_0 V(B^n_\infty).
\end{align*}
Lemma \ref{closesimplex} implies that $
\widetilde{B}^n_\infty=\{x\in\R^n:|\langle x,u_i\rangle|\le 1,\; i=1,\ldots,n\}$ 
satisfies 
$$H^+\cap B^n_\infty\subset \left(H^+\cap\widetilde{B}^n_\infty\right)\cup \left(B^n_\infty\setminus (1-\sqrt{n}\eta)B^n_\infty\right) $$
and therefore
\begin{align*}
    V(H^+\cap \widetilde{B}^n_\infty)&\ge V(H^+\cap B^n_\infty)-\left(1-(1-\sqrt{n}\eta)^n\right)V(B^n_\infty)\\
    &\ge \left(2^{-3}n^{-\frac{1}{2}}n^{-n}\omega_0-\sqrt{n}n\eta\right)V(B^n_\infty).
\end{align*}
Another application of Lemma \ref{closesimplex} yields
\begin{align*}
    V(P)&\le V(\widetilde{B}^n_\infty)-V(H^+\cap \widetilde{B}^n_\infty)\\
    &\le V((1+2\sqrt{n}\eta)B^n_\infty)
    -\left(2^{-3}n^{-\frac{1}{2}}n^{-n}\omega_0-\sqrt{n}n\eta\right)V(B^n_\infty)\\
    &\le \left(1+2^n\sqrt{n}n\eta+\sqrt{n}n\eta-2^{-3}n^{-\frac{1}{2}}n^{-n}\omega_0\right)V(B^n_\infty)\\
    &\le \left(1+2^n n^2\eta-2^{-3}n^{-\frac{1}{2}}n^{-n}\omega_0\right)V(B^n_\infty)\\
    &\le \left(1+\frac{2^n n^2}{c(n)}\omega_0-2^{-3}n^{-\frac{1}{2}}n^{-n}\omega_0\right)V(B^n_\infty)\\
    &\le \left(1 -2^{-4}n^{-\frac{1}{2}}n^{-n}\omega_0\right)V(B^n_\infty),
\end{align*}
which implies the asserted volume bound.
    \hfill\proofbox
\end{proof}

\medskip 

After these preparations, parts (b) and (c) of Theorem ~\ref{meanw-iso-measure-stab} follow from the following theorem.

\begin{theo}\label{theo:theorem 7.7}
Let $\mu$ be an even isotropic measure on $S^{n-1}$.   If 
\begin{equation}\label{eq:A7.0}
    \ell(Z_\infty^*(\mu))\le (1+\varepsilon)\ell(Z_\infty^*(\nu)),
\end{equation}
for some $\varepsilon\in (0,\varepsilon_0)$ with $\varepsilon_0=2^{-70}n^{-30n}$, then 
$$
\delta_{\rm HO}({\rm supp}\,\mu,{\rm supp}\,\nu)\le 2^{20}n^{13n}\varepsilon^{\frac{1}{4}}.
$$
\end{theo}

\begin{proof} For the given measure $\mu$, we choose an even discrete isotropic measure on $S^{n-1}$ according to Lemma \ref{JohnLimited} with ${\rm supp}\, \mu_0=\{u_1,\ldots,u_k\}\subset {\rm supp}\, \mu$. In particular, we then have $Z^*_\infty(\mu)\subset Z^*_\infty(\mu_0)$ and 
\begin{equation}\label{eq:A7.1}
\ell(Z^*_\infty(\mu_0))\le (1+\varepsilon)\ell(Z^*_\infty(\nu)).
\end{equation}
We claim that
\begin{equation}\label{eq:A7.2}
\gamma_n(b Z^*_\infty(\mu_0))\ge \gamma_n(b Z^*_\infty(\nu))-\varepsilon \gamma_n( Z^*_\infty(\nu))
\end{equation}
for some $b\in [1,2]$. 

For the proof, assume to the contrary that 
$$
\gamma_n(b Z^*_\infty(\mu_0))< \gamma_n(b Z^*_\infty(\nu))-\varepsilon \gamma_n( Z^*_\infty(\nu))
$$
or all $b\in [1,2]$. By Proposition \ref{gammanbP-prop}, 
$$
\gamma_n(t Z^*_\infty(\mu_0))\le \gamma_n(t Z^*_\infty(\nu))\quad \text{for }t>0.
$$
Hence
\begin{align*}
\ell(Z_\infty^*(\nu))-\ell(Z^*_\infty(\mu_0))&=\int_0^\infty\left(\gamma_n(t Z^*_\infty(\mu_0))- \gamma_n(t Z^*_\infty(\nu))\right)\, \dint t\\
&<-\varepsilon\ell(Z_\infty^*(\nu)),    
\end{align*}
which contradicts \eqref{eq:A7.1}. 

Thus relation \eqref{gammabPeps-stab} holds for some $b\in[1,2]$ with $\varepsilon$ replaced by $\varepsilon \ell(Z_\infty^*(\nu))$. In the proof of relation \eqref{gammabP-stabvoleps} it is shown that there exists an orthonormal basis $w_1,\ldots,w_n$ of $\R^n$ such that (w.l.o.g.) 
$$
\angle(u_i,w_i)<c_1(n)\varepsilon ^{\frac{1}{4}}\quad\text{for }i=1,\ldots,n
$$
with $c_1(n):=3n2^{14}n^{4n}n^{\frac{1}{8}}$, where we used that $\ell(Z_\infty^*(\nu))=\frac{1}{2}\ell(B^n_2)W(B^n_1)\le \ell(B^n_2)\le \sqrt{n}$ (cf.~the proof of \cite[Lemma 5.7]{BFH21}). If $\nu$ is the cross measure associated with $\pm w_1,\ldots,\pm w_n$, then
\begin{align*}
    {\rm supp}\,\nu&\subset {\rm supp}\,\mu_0 +c_1(n)\varepsilon^{\frac{1}{4}}B^n_2
    \subset {\rm supp}\,\mu +2^{20}n^{13n}\varepsilon^{\frac{1}{4}}B^n_2.
\end{align*}
It remains to be shown that
\begin{equation}\label{eq:A7.3}
      {\rm supp}\,\mu \subset {\rm supp}\,\nu +2^{20}n^{13n}
      \varepsilon^{\frac{1}{4}}B^n_2.
\end{equation}
For the proof, we set $\eta=c_1(n) \varepsilon^{\frac{1}{4}}$. Clearly, $\pm u_i$ is contained in the set on the right-hand side of \eqref{eq:A7.3} for $i\in\{1,\ldots,n\}$. Recall that $c(n)=2^{n+4}n^{n+3}$ and define
$$
\Delta=\max_{u\in{\rm supp}\, \mu}\;\min_{i=1,\ldots,n}\angle(u,w_i). 
$$
If $\Delta<c(n)\eta$, then relation \eqref{eq:A7.3} holds. Henceforth we consider the case where $\Delta\ge c(n)\eta$. Then we may assume that the maximum in the definition of $\Delta$ is realized by $u_0\in ({\rm supp}\, \mu)\setminus\{\pm u_1,\ldots,\pm u_n\}$. We define
\begin{align*}
P&=\{x\in\R^n:|\langle x,u_i\rangle|\le 1 ,\; i=0,1,\ldots,n\},\\
\widetilde{B}^n_\infty &=\{x\in\R^n:|\langle x,u_i\rangle|\le 1 ,\; i=1,\ldots,n\},\\  
 {B}^n_\infty&=\{x\in\R^n:|\langle x,w_i\rangle|\le 1 ,\; i=1,\ldots,n\},
\end{align*}
hence $Z^*_\infty(\mu)\subset P\subset \widetilde{B}^n_\infty$. 
By Lemma \ref{volumebound} we have
$$
V(P)\le \left(1-\frac{\Delta}{2^4n^{n+1}}\right)V(B^n_\infty),
$$
hence 
$$
\delta_{\rm vol}(P,B^n_\infty)\ge \frac{\Delta}{2^4n^{n+1}} V(B^n_\infty).
$$
Lemma \ref{closesimplex} implies that
\begin{align*}
\delta_{\rm vol}(B^n_\infty,\widetilde{B}^n_\infty)
&\le \left[(1+2\sqrt{n}\eta)^n-(1-\sqrt{n}\eta)^n\right]V(B^n_\infty)\le 4\sqrt{n} n\eta V(B^n_\infty),
\end{align*}
since $(1+2\sqrt{n}\eta)^n\le 4/3$. The triangle inequality yields
\begin{align}\label{eq:A7.4}
V(\widetilde{B}^n_\infty\setminus P)&=\delta_{\rm vol}(\widetilde{B}^n_\infty,P)\ge 
\delta_{\rm vol}({B}^n_\infty,P)-\delta_{\rm vol}({B}^n_\infty,\widetilde{B}^n_\infty)\nonumber\\
&\ge \left(\frac{\Delta}{2^4n^{n+1}} -4\sqrt{n}n\eta\right)V(B^n_\infty).    
\end{align}
We claim that 
\begin{equation}\label{eq:A7.5}
V(\widetilde{B}^n_\infty\setminus P)\le 2^{17} n^{11n}\varepsilon^{\frac{1}{4}}.
\end{equation}
Combination of \eqref{eq:A7.4} and \eqref{eq:A7.5} shows that
\begin{align*}
    \Delta&\le 2^4n^{n+1}\left(4\sqrt{n}n\eta+2^{-n}2^{17} n^{11n}\varepsilon^{\frac{1}{4}}\right)\le 2^{20}n^{13n}\varepsilon^{\frac{1}{4}},
\end{align*}
which completes the proof of \eqref{eq:A7.3}, once \eqref{eq:A7.5} has been established.

We finally verify \eqref{eq:A7.5}. Using Lemma \ref{closesimplex}, we get $\widetilde{B}^n_\infty\subset (1+2\sqrt{n}\eta)B^n_\infty\subset 2B^n_\infty$, hence $1/(2\sqrt{n})\widetilde{B}^n_\infty\subset B^n_2$. Then
\begin{align}\label{eq:A7.6}
\ell(Z_\infty^*(\mu))-\ell(\widetilde{B}^n_\infty ))&=\int_0^\infty\left(\gamma_n(t \widetilde{B}^n_\infty)- \gamma_n(t Z^*_\infty(\mu))\right)\, \dint t\nonumber\\
&\ge \int_0^{1/(2\sqrt{n})}\frac{e^{-1/2}}{(2\pi)^{n/2}}t^nV(\widetilde{B}^n_\infty\setminus Z^*_\infty(\mu))\, \dint t\nonumber\\
&\ge \frac{1}{n+1}\left(\frac{1}{2\sqrt{n}}\right)^{n+1}  \frac{e^{-1/2}}{(2\pi)^{n/2}}
V(\widetilde{B}^n_\infty\setminus Z^*_\infty(\mu))\frac{\ell(B^{n}_\infty)}{\sqrt{n}}\nonumber\\
&\ge n^{-6n}V(\widetilde{B}^n_\infty\setminus Z^*_\infty(\mu))\ell(B^{n}_\infty).
\end{align}
Again by an application of Lemma \ref{closesimplex}, we obtain
\begin{align}\label{eq:A7.7}
 \ell(\widetilde{B}^n_\infty ))-\ell( {B}^n_\infty )&\ge 
 \ell((1+2\sqrt{n}\eta){B}^n_\infty ))-\ell( {B}^n_\infty )\nonumber\\
 &=\left[(1+2\sqrt{n}\eta)^{-1}-1\right]\ell( {B}^n_\infty )
 \nonumber\\
 &\ge -2\sqrt{n}\eta \ell( {B}^n_\infty ).
\end{align}
Combining \eqref{eq:A7.0}, \eqref{eq:A7.6} and \eqref{eq:A7.7}, 
we arrive at
\begin{align*}
\varepsilon \ell( {B}^n_\infty )&\ge \ell(Z^*_\infty(\mu) )-\ell( {B}^n_\infty )
\ge \left(\frac{V(\widetilde{B}^n_\infty\setminus Z^*_\infty(\mu))}{n^{6n}}-2\sqrt{n}\eta\right)\ell( {B}^n_\infty ),
\end{align*}
hence
\begin{align*}
   V(\widetilde{B}^n_\infty\setminus Z^*_\infty(\mu))&\le n^{6n}(\varepsilon+2\sqrt{n}\eta)
   \le n^{6n}\left(1+2\sqrt{n}3n2^{14}n^{4n}n^{\frac{1}{8}}\right)\varepsilon^{\frac{1}{4}}\\
   &\le 2^{17}n^{11n}\varepsilon^{\frac{1}{4}},
\end{align*}
which proves the claim, and thus the theorem.
\hfill\proofbox
\end{proof}

\bigskip

\noindent
Authors' addresses:

\medskip

\noindent
K\'aroly J. B\"or\"oczky, MTA Alfr\'ed R\'enyi Institute of Mathematics, Hungarian Academy of Sciences, Re\'altanoda u. 13-15, 1053 Budapest, Hungary. E-mail: carlos@renyi.hu

\medskip
\noindent
Ferenc Fodor, Bolyai Institute, University of Szeged, Aradi v\'ertan\'uk tere 1, 6720 Szeged, Hungary. E-mail: fodorf@math.u-szeged.hu

\medskip
\noindent
Daniel Hug, Karlsruhe Institute of Technology (KIT), D-76128 Karlsruhe, Germany. E-mail: daniel.hug@kit.edu

\end{document}